\RequirePackage{amsthm}
\documentclass[sn-mathphys,Numbered]{sn-jnl}

\usepackage{graphicx}%
\usepackage{multirow}%
\usepackage{amsmath,amssymb,amsfonts}%
\usepackage{amsthm}%
\usepackage{mathrsfs}%
\usepackage[title]{appendix}%
\usepackage{xcolor}%
\usepackage{manyfoot}%
\usepackage{booktabs}%
\usepackage{listings}%

\usepackage{float}
\usepackage{multicol}
\usepackage[all,color]{xy}
\usepackage{bbold}
\usepackage{stix}
\usepackage{geometry}

\newtheorem{thm}{Theorem}[subsection]
\newtheorem{prop}[thm]{Proposition}
\newtheorem{cor}[thm]{Corollary}
\newtheorem{lem}[thm]{Lemma}
\theoremstyle{definition}
\newtheorem{defn}[thm]{Definition}
\theoremstyle{remark}
\newtheorem{ex}[thm]{Example}
\newcommand{\cat}[1]{\mathfrak{#1}}

\DeclareMathOperator{\ob}{Ob}

\DeclareMathOperator{\Del}{Del}
\DeclareMathOperator{\sgn}{sgn}

\newcommand{\Set}{\mathbf{Set}}

\begin{document}

\title[Incidence Hypergraphs]{Incidence Hypergraphs: Box Products and the Laplacian}

\author[1]{\fnm{Will} \sur{Grilliette}}\email{wbgrill@nsa.gov}
\affil[1]{\orgname{National Security Agency}, \orgaddress{\city{Fort George G Meade MD}, \postcode{20755-6844}, \state{MD}, \country{USA}}}

\author*[2]{\fnm{Lucas J.} \sur{Rusnak}}\email{Lucas.Rusnak@txstate.edu}
\affil*[2]{\orgdiv{Texas State University}, \orgname{Department of Mathematics}, \orgaddress{\street{601 University Drive}, \city{San Marcos}, \postcode{78666}, \state{TX}, \country{USA}}}

\equalcont{These authors contributed equally to this work.}

\abstract{The box product and its associated box exponential are characterized for the categories of quivers (directed graphs), multigraphs, set system hypergraphs, and incidence hypergraphs. It is shown that only the quiver case of the box exponential can be characterized via homs entirely within their own category. An asymmetry in the incidence hypergraphic box product is rectified via an incidence dual-closed generalization that effectively treats vertices and edges as real and imaginary parts of a complex number, respectively. This new hypergraphic box product is shown to have a natural interpretation as the canonical box product for graphs via the bipartite representation functor, and its associated box exponential is represented as homs entirely in the category of incidence hypergraphs; with incidences determined by incidence-prism mapping. The evaluation of the box exponential at paths is shown to correspond to the entries in half-powers of the oriented hypergraphic signless Laplacian matrix.}



\maketitle

\setcounter{tocdepth}{2}
\tableofcontents

\section{Introduction}

The box product of two graphs $G_1$ and $G_2$, denoted $G_1 \square G_2$, is a graph with vertex set $V(G_1) \times V(G_2)$ where two vertices $(u_1,u_2)$ and $(w_1,w_2)$ are adjacent via an edge in $G_1 \square G_2$ if and only if, for $i \neq j$, $u_i = w_i$ and $u_j$ is adjacent to $w_j$. This produces a copy of $G_1$ at every vertex of $G_2$ and a copy of $G_2$ at every vertex of $G_1$; effectively turning each adjacency pair into a ``box,'' or $4$-cycle in $G_1 \square G_2$. The box product $G \square P_1$, consisting of a graph $G$ with a path of length $1$, is called the prism of $G$; while specializing $G$ to a path of length $k$ is a ladder. Figure \ref{fig:PrismLadder} depicts these two box products as well as its namesake single adjacency product; being a closed symmetric monoidal product, commutativity, associativity and exponents are implied. 

\begin{figure}[H]
    \centering
    \includegraphics{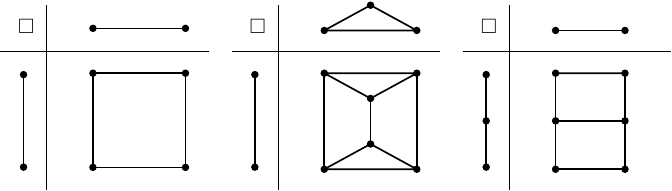}
    \caption{Left: Two edges forming a box. Middle: A $K_3$ prism. Right: A ladder with height $2$.}
    \label{fig:PrismLadder}
\end{figure}

Box products have applications in the study of chromatic numbers where $\chi_{G \square H} = \max\{\chi_{G},\chi_{H}\}$ \cite{Sabidussi1}, and domination numbers where $\gamma_{G \square H} = \max\{\gamma_{G},\gamma_{H}\}$ \cite{Vizing1}, and recognizing a box product has been shown to be in polynomial time \cite{IMRICH2007472}. Specific evaluations of box products have garnered their own attention where $G \square P_1$ has been used to study prism-hamiltonicity \cite{PrismHam2,PrismHam3,PrismHam1}; where being prism-hamiltonian is stronger than having a $2$-walk and weaker than having a hamilton path. In terms of morphisms, proper colorings are mappings into a smallest complete graph, hamiltonicity is the existence of a monic map of a circuit of length $\lvert V \rvert$, and powers of both the Laplacian and adjacency matrices are unified by mapping paths into the underlying incidence structure simply by adjusting incidence-monic-ness \cite{OHHar,AH1}. While a counterexample was recently produced \cite{HedetniemiFalse}, Hedetniemi conjectured that $\chi(G \times H)=\min\{\chi(G),\chi(H)\}$, which is equivalent to every complete graph being multiplicative in the category of simple graphs. It was shown \cite{El_Zahar_1985} that the multiplicativity of $K$ is equivalent to either $G$ or the exponential graph $K^G$ being $K$-colorable, for every graph $G$, reinforcing the necessity to study exponentials and strive for a natural approach to homs.

Graph products \cite{handbookpg,ImrichPGSR}, graph transformations \cite{graphtransformation,graphgrammars3,GandH}, and categorical graph theory \cite{brown2008, bumby1986, dorfler1980, raeburn, schiffler} provide a natural approach and a unified context to study graphs along with their mapping properties. The box and categorical products of simple graphs are two graph products that turn the category of simple graphs into a closed symmetric monoidal category.  Unfortunately, the exponential structure does not extend pleasantly to parallel edges, as shown in \cite{IH1}, the category of multigraphs $\cat{M}$ is not cartesian (categorical product) closed.  We demonstrate that a box exponential does exist for both $\cat{M}$ and the category $\cat{H}$ of set-system hypergraphs, but the structure is quite peculiar as seen in Definition \ref{box-exponential-h}.  This strange structure seems to be due to the failure of the edge functor to admit a left adjoint, much like the failure of the cartesian structure  in \cite{IH1}.  The work in \cite{IH1} remedied the cartesian issue by introducing the category of incidence hypergraphs $\cat{R}$, and the current work seeks the same goal for the box product and its internal hom structure.

To that end, in Section \ref{sec:the4boxes} we characterize and generalize box products and their exponentials for the following four graph-like categories: (1) the category of quivers $\mathfrak{Q}$ (directed graphs); (2)  the category of set-system hypergraphs $\mathfrak{H}$; (3) the category of multigraphs $\mathfrak{M}$; and (4) the category of incidence hypergraphs $\mathfrak{R}$. Each are shown to be a closed symmetric monoidal category with respect to the corresponding box product. These provide a characterization of ``cartesian exponentials'' and ``diamond products'' related to the examination of hom-complexes \cite{DOCHTERMANN2009490,kilp2001,knauer-book}. In Section \ref{sec:LapProd} we introduce a new dually-closed incidence-theoretic box product and exponential for hypergraphs called the \emph{Laplacian product (exponential)}. The Laplacian product $G \blacksquare H$ effectively treats vertices and edges as real and imaginary parts of a complex number --- that is, and edge times and edge will be a vertex. When evaluated at paths, it provides a dual treatment of prisms and ladders that blur the concepts of hamiltonian and eulerian. The Laplacian exponential $[G,H]_{L}$ is shown to consist entirely of $\cat{R}$-homs with vertices, edges, and incidences of the Laplacian exponential as follows:
\begin{align*}
\check{V}[G,H]_{L}&=\cat{R}(G,H), \\
\check{E}[G,H]_{L}&=\cat{R}\left(G^{\#},H\right)\cong\cat{R}\left(G,H^{\#}\right), \\
I[G,H]_{L} &=\cat{R}\left(G\blacksquare \check{P}_{1/2},H\right).
\end{align*}
The incidences are determined by the incidence-prism $G\blacksquare \check{P}_{1/2}$ mapping into $H$. Section \ref{sec:upsilon} demonstrates the naturalness by providing a graphical characterization of the Laplacian product to the box product of bipartite representation graph. Exponential hypergraphs of the form $[\check{P}_{k/2},G]_{L}$
are combinatorially characterized via the weak-walk interpretation of the Laplacian from \cite{OHHar,IH2,AH1}, and shown that the vertices and edges are entries in the $k^{\text{th}}$ power of the Laplacian. We also observe that the incidences produced by incidence-prism embeddings produce the matrix minor framework to place the Laplacian entries. Since the incidences are naturally contravariant to the prism-hamiltonian problem, we leave the question open as to what is more natural to specialize to generalized notions of (prism)-hamiltonicity; the incidences of the exponential or direct evaluations of the form $[\check{C}_{k},G]_{L}$. Since $\cat{R}$-homs of the form $\cat{R}(\check{C}_{k},G)$ are vertices in the Laplacian exponential, the existence of a specific vertex in $[\check{C}_{k},G]_{L}$ corresponding to a monic closed path (cycle) embedding would tell us if $G$ was hamiltonian. Replacing $G$ with $G\blacksquare \check{P}_{1/2}$ would be a generalization of prism-hamiltonian.

\subsection{Incidence Hypergraphs}

The category of incidence hypergraphs $\cat{R}$ introduced in \cite{IH1} serves to remedy some peculiar phenomena in preexisting categories of graph-like objects, while injectivity and its relation to hypergraphic uniformity was characterized in \cite{IH2}, with applications to characteristic polynomials \cite{OHSachs,OHMTT} and their connection to Hadamard matrices \cite{Reff3,Reff5}. This section is a condensed version of \cite{IH1} necessary to discuss the objects in the category of incidence hypergraphs. A formal statement and glossary to provide quick combinatorial context are included in the Appendix.

An incidence hypergraph is a quintuple $G=(\check{V}, \check{E}, I, \varsigma, \omega)$ consisting of a set of vertices $\check{V}$, a set of edges $\check{E}$, a set of incidences $I$, and two incidence maps $\varsigma:I\to\check{V}$, and $\omega:I\to\check{E}$. Note, this notation is from \cite{IH1}, where the set decorations distinguish between the functors into $\Set$ for different graph-like categories; for example, $\check{V}(G)$ is the set of vertices of an incidence hypergraph, $\vec{V}(G)$ is the set of vertices of a quiver, and $V(G)$ is the set of vertices of a graph. A \emph{directed path of length $n/2$} is a non-repeating sequence 
\begin{equation*}
\check{P}_{n/2}=(a_{0},i_{1},a_{1},i_{2},a_{2},i_{3},a_{3},...,a_{n-1},i_{n},a_{n})
\end{equation*}
of vertices, edges, and incidences, where $\{a_{\ell }\}$ is an alternating sequence of vertices and edges, and $i_{j}$ is an incidence between $a_{j-1}$ and $a_{j}$. The \emph{tail} of a path is $a_0$ and the \emph{head} of a path is $a_n$. In terms of paths, the generators of $\cat{R}$ are the path of length zero consisting of a single vertex, the path of length zero consisting of a single edge, and the $1$-edge. The $1$-edge generator is critical to the structure theorems as it is also terminal. To denote a $1$-edge we use $\check{P}_{1/2}$ when we want to emphasize the vertex to edge path nature of the object, and we use $I^{\diamond}(\{1\})$ when we want to emphasize that it is also the left adjoint of the incidence functor on a singleton set.

It was shown in \cite{IH1} that there is a logical functor from the category of quivers to incidence hypergraphs $\xymatrix{\cat{Q}\ar[r]^{\Upsilon} & \cat{R}}$  that characterizes the quiver exponential entirely as hom-sets from $\cat{R}$. The left adjoint $\Upsilon ^{\diamond}$ produces the bipartite incidence quiver, and when composed with the undirecting functor $U$, $U\Upsilon ^{\diamond}$ is the canonical bipartite representation of a hypergraph.

\begin{figure}[H]
    \centering
    \includegraphics{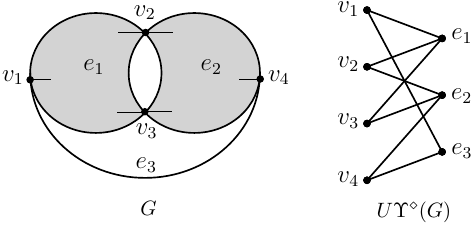}\\
    \caption{An incidence hypergraph and its bipartite representation via functors.}
    \label{fig:GandBipartite}
\end{figure}

We demonstrate that $U\Upsilon ^{\diamond}$ is a strong symmetric monoidal functor that links the graph box product of bipartite representations with a new dual-closed incidence box product. Moreover, path evaluations of the exponential (right adjoint) of this new box product has a combinatorial interpretation as half powers of the signless hypergraphic Laplacian.

\subsection{Oriented Hypergraphs}
\label{ssec:OHs}

An \emph{orientation} of an incidence hypergraph is a function $\sigma:I\rightarrow\{+1,-1\}$. While signed graphs and incidence orientations are have their foundations in \cite{MR0267898,OSG}, the concept was extended to hypergraphs in \cite{AH1, OH1}. Spectral properties and eigenvalue bounds of oriented hypergraphs have been studied in \cite{Mulas4,Reff2,Reff4,Wang1}, while various characteristic polynomials were characterized via incidence-path mapping families in \cite{OHSachs, OHMTT}, generalizing the work in \cite{Sim1, Seth1}.

The matrices commonly associated to algebraic graph theory have oriented hypergraphic analogs, and have been combinatorially classified via weak walks in \cite{AH1} via path embeddings. A \emph{directed weak walk of $G$} is the image of an incidence-preserving map of a directed path into $G$. A \emph{backstep of }$G$ is a non-incidence-monic map of $\check{P}_{1}$ into $G$; a \emph{loop of }$G$ is an incidence-monic map of $\check{P}_{1}$ into $G$ that is not vertex-monic; and a \emph{directed adjacency of }$G$ is a map of $\check{P}_{1}$ into $G$ that is incidence-monic. The \emph{sign of a weak walk} $W$ is 
\begin{equation*}
\sgn(W)=(-1)^{\lfloor n/2\rfloor }\prod_{h=1}^{n}\sigma (i_{h}).
\end{equation*} 
The \emph{incidence matrix} of an oriented hypergraph $G$ is the $V \times E$ matrix $\mathbf{H}_{G}$ where the $(v,e)$-entry is the sum of $\sigma(i)$ for each $i \in I$ such that $\varsigma (i)=v$ and $\omega (i)=e$. The \emph{adjacency matrix} $\mathbf{A}_{G} $ of an oriented hypergraph $G$ is the $V\times V$ matrix whose $(u,w)$-entry is the sum of $\sgn(q(\check{P}_{1}))$ for all incidence monic maps $q:\check{P}_{1}\rightarrow G$ with $q(\varsigma(i_1))=u$ and $q(\varsigma(i_2))=w$. The \emph{degree matrix} of an oriented hypergraph $G$ is the $V\times V$ diagonal matrix whose $(v,v)$-entry is the sum of all non-incidence-monic maps $p:\check{P}_{1}\rightarrow G$ with $p(\varsigma(i_1))=p(\varsigma(i_2)=v$. The \emph{Laplacian matrix of $G$} is defined as $\mathbf{L}_{G}:=\mathbf{H}_{G} \mathbf{H}_{G}^{T}=\mathbf{D}_{G}-\mathbf{A}_{G}$ for all oriented hypergraphs. These definitions are a result of the path-embedding weak-walk theorem that was implied in \cite{OHHar, AH1} and collected in \cite{OHSachs}.

\begin{thm}[\cite{OHSachs}, Theorem 2.3.1]
\label{t:WWT}
Let $G$ be an oriented hypergraph.

\begin{enumerate}
\item The $(v,w)$-entry of $\mathbf{D}_{G}$ is the number of strictly weak, weak walks, of length $1$ from $v$ to $w$. That is, the number of backsteps from $v$ to $w$.

\item The $(v,w)$-entry of $\mathbf{A}_{G}$ is the number of positive (non-weak) walks of length $1$ from $v$ to $w$ minus the number of negative (non-weak) walks of length $1$ from $v$ to $w$.

\item The $(v,w)$-entry of $-\mathbf{L}_{G}$ is the number of positive weak walks of length $1$ from $v$ to $w$ minus the number of negative weak walks of length $1$ from $v$ to $w$.
\end{enumerate}
\label{WWL}
\end{thm}

Moreover, from \cite{OHHar} this holds for $k^{\text{th}}$ powers of these matrices via paths of length $k$. Combining the incidence matrices of $G$ and its dual $G^{\#}$ into a single $\left( \lvert V \rvert + \lvert E \rvert \right) \times \left( \lvert V \rvert + \lvert E \rvert \right)$ incidence matrix define the 
\emph{complete incidence matrix} as
\begin{align*}
\overline{\mathbf{H}}_{G}:=\left[ 
\begin{array}{cc}
\mathbf{0} & \mathbf{H}_{G} \\ 
\mathbf{H}_{G^{\#}} & \mathbf{0}
\end{array}
\right] =\left[ 
\begin{array}{cc}
\mathbf{0} & \mathbf{H}_{G} \\ 
\mathbf{H}_{G}^{T} & \mathbf{0}
\end{array}
\right],
\end{align*}
and the \emph{complete Laplacian}  as
\begin{align*}
\overline{\mathbf{L}}_{G}:=\overline{\mathbf{H}}_{G}^{2}=\left[ 
\begin{array}{cc}
\mathbf{L}_{G} & \mathbf{0} \\ 
\mathbf{0} & \mathbf{L}_{G^{\#}}
\end{array}
\right].
\end{align*}
Thus, the entries of $\overline{\mathbf{L}}_{G}$ are determined by the morphisms $\cat{R}(\check{P}_{k/2},G)$ and $\cat{R}(\check{P}^{\#}_{k/2},G)$, and the weak-walk theorem can be restated as follows.

\begin{thm}
Let $G$ be an oriented hypergraph and $k\in \mathbb{Z}_{\geq 0}$, $\overline{\mathbf{H}}_{G}^{k}=(-1)^{\lfloor k/2\rfloor }\overline{\mathbf{L}}_{G}^{k/2}$. Moreover, the incidence signing function for objects in $\cat{R}$ is edge signing in $\cat{M}$ under $U\Upsilon^{\diamond},$ thus these matrices are also equal to the standard signed graphic adjacency matrix $\mathbf{A}_{U\Upsilon^{\diamond}(G)}^{k}$ with the inherited edge signing function.
\label{t:LapExpMakesLap}
\end{thm}

In Section \ref{sec:LapProd} we prove that the Laplacian is simply an evaluation of the Laplacian exponential at paths. Specifically, the 
vertices, edges, and incidences are shown to be the following $\cat{R}$-homs:
\begin{align*}
\check{V}[\check{P}_{k/2},G]_{L}&=\cat{R}(\check{P}_{k/2},G), \\
\check{E}[\check{P}_{k/2},G]_{L}&=\cat{R}\left(\check{P}_{k/2}^{\#},G\right)\cong\cat{R}\left(\check{P}_{k/2},G^{\#}\right), \\
I\left[\check{P}_{k/2},G\right]_{L}
&=\cat{R}\left(\check{P}_{k/2}\blacksquare \check{P}_{1/2},G\right),
\end{align*}
where the incidences are determined by the dualized-ladder mapping of $\check{P}_{k/2}\blacksquare \check{P}_{1/2}$.

\section{Box Products for Graph-like Categories}
\label{sec:the4boxes}

In this section we provide a short categorical development of box products and exponentials on the categories of quivers $\mathfrak{Q}$, set-system hypergraphs $\mathfrak{H}$, multigraphs $\mathfrak{M}$, and incidence hypergraphs $\mathfrak{R}$. These have direct applications to hom complexes of graphs and graph products \cite{DOCHTERMANN2009490, DOCHTERMANN2009180, handbookpg, Knauer1990,KNAUER2001592}; the simple $\mathfrak{M}$ box product matches, and the $\mathfrak{H}$ exponential generalizes, the ``diamond products'' in \cite{DOCHTERMANN2009490,knauer-book}. The structure maps for each monoidal structure appear in Section \ref{StrMaps}. The necessary left adjoints of appear throughout each subsection. The left adjoint of the vertex functors at a singleton set ($\vec{V}^{\diamond}(\{1\})$, ${V}^{\diamond}(\{1\})$, $\check{V}^{\diamond}(\{1\})$) are just an isolated vertex in their category, the left adjoint of the edge functors at a singleton set are single edges (with necessary vertex structure). However, as shown in \cite{IH1}, the edge functors for multigraphs and set-system hypergraphs do not have left adjoints, which is caused by the covariant power-set functor in their construction and leads to the cartesian structure failing.  Moreover, this failure appears to contribute to the peculiarities in the box exponential.

\subsection{Box Products for Quivers}
\label{CatQ}
Recall, that a quiver $Q$ consists of a set of vertices $\vec{V}(Q)$, a set of edges $\vec{E}(Q)$, and a source and target map $\sigma_{Q}$ and $\tau_{Q}$, respectively. The action of the box product on quivers is well known in sources such as \cite{handbookgt,handbookpg}.  This action can be naturally extended to quiver homomorphisms to create a symmetric monoidal product with the structure maps from Section \ref{StrMaps}.

\begin{defn}[Box product]
Given quivers $Q$ and $P$, define the quiver $Q\vec{\Box}P$ by
\begin{enumerate}
\item $\vec{V}\left(Q\vec{\Box}P\right):=\vec{V}(Q)\times\vec{V}(P)$,
\item $\vec{E}\left(Q\vec{\Box}P\right):=\left(\{1\}\times\vec{E}(Q)\times\vec{V}(P)\right)\cup\left(\{2\}\times\vec{V}(Q)\times\vec{E}(P)\right)$,
\item $\sigma_{Q\vec{\Box}P}(n,x,y):=\left\{\begin{array}{cc}
\left(\sigma_{Q}(x),y\right)	&	n=1,\\
\left(x,\sigma_{P}(y)\right)	&	n=2,\\
\end{array}\right.$
\item $\tau_{Q\vec{\Box}P}(n,x,y):=\left\{\begin{array}{cc}
\left(\tau_{Q}(x),y\right)	&	n=1,\\
\left(x,\tau_{P}(y)\right)	&	n=2.\\
\end{array}\right.$
\end{enumerate}
Given quiver homomorphisms $\xymatrix{Q_{1}\ar[r]^{\phi} & Q_{2}}$ and $\xymatrix{P_{1}\ar[r]^{\psi} & P_{2}}$, define $\xymatrix{Q_{1}\vec{\Box}P_{1}\ar[r]^{\phi\vec{\Box}\psi} & Q_{2}\vec{\Box}P_{2}}$ by
\begin{enumerate}
\item $\vec{V}\left(\phi\vec{\Box}\psi\right)(v,w):=\left(\vec{V}(\phi)(v),\vec{V}(\psi)(w)\right)$,
\item $\vec{E}\left(\phi\vec{\Box}\psi\right)(n,x,y):=\left\{\begin{array}{cc}
\left(1,\vec{E}(\phi)(x),\vec{V}(\psi)(y)\right)	&	n=1,\\
\left(2,\vec{V}(\phi)(x),\vec{E}(\psi)(y)\right)	&	n=2.\\
\end{array}\right.$
\end{enumerate}
\end{defn}

\begin{ex}
We calculate the object formed by the $\cat{Q}$-box product, consider the single directed edge $\vec{P}_1\cong\vec{E}^{\diamond}(\{1\})$ (the left adjoint consists of disjoint edges on the given set). The quiver box product of two directed edges appears in Figure \ref{fig:QBoxProd}.
\begin{figure}[H]
    \centering
    \includegraphics{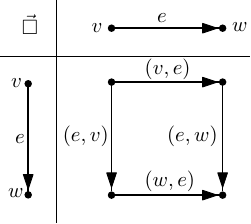}\\
    \caption{The quiver box product of $\vec{P}_1\vec{\Box}\vec{P}_1$.}
    \label{fig:QBoxProd}
\end{figure}
It is unsurprising that the directed box product is related to the canonical box product on graphs where the undirecting functor $U$ returns Figure \ref{fig:PrismLadder} (Left).
\end{ex}

The construction of the quiver box exponential $\left[Q_{1},Q_{2}\right]_{B}$ is straight forward via the following isomorphisms, much like the categorical exponential from \cite[Definition 3.49]{IH1}.
\begin{align*}
\vec{V}\left[Q_{1},Q_{2}\right]_{B}
&\cong\Set\left(\{1\},\vec{V}\left[Q_{1},Q_{2}\right]_{B}\right)
\cong\cat{Q}\left(\vec{V}^{\diamond}(\{1\}),\left[Q_{1},Q_{2}\right]_{B}\right) \\
&\cong\cat{Q}\left(Q_{1}\vec{\Box}\vec{V}^{\diamond}(\{1\}),Q_{2}\right)
\cong\cat{Q}\left(Q_{1},Q_{2}\right), \\
\vec{E}\left[Q_{1},Q_{2}\right]_{B}
&\cong\Set\left(\{1\},\vec{E}\left[Q_{1},Q_{2}\right]_{B}\right)
\cong\cat{Q}\left(\vec{E}^{\diamond}(\{1\}),\left[Q_{1},Q_{2}\right]_{B}\right)
\cong\cat{Q}\left(Q_{1}\vec{\Box}\vec{E}^{\diamond}(\{1\}),Q_{2}\right).
\end{align*}

For the source and target maps, the Yoneda embedding will be helpful.  This important functor arises naturally from the presheaf structure of $\cat{Q}$ as seen in \cite[Example I.1.4.3.a]{borceux}, and the characterization below follows from direct calculation.  For context, let $\cat{E}$ be the finite category drawn below.
\[\xymatrix{
1\ar@/^/[r]^{s}\ar@/_/[r]_{t}	&	0\\
}\]
Then, $\cat{Q}=\Set^{\cat{E}}$, and $\xymatrix{\Set & \cat{Q}\ar[l]_(0.4){\vec{V}}\ar[r]^(0.4){\vec{E}} & \Set}$ are the evaluation functors at $0$ and $1$, respectively. 

\begin{prop}[Yoneda functor]
Let $Y_{\cat{Q}}:\cat{E}^{\mathrm{op}}\to\cat{Q}$ be the Yoneda embedding.  Then, $Y_{\cat{Q}}(0)\cong\vec{V}^{\diamond}(\{1\})$ and $Y_{\cat{Q}}(1)\cong\vec{E}^{\diamond}(\{1\})$.  Moreover, $\xymatrix{Y_{\cat{Q}}(0)\ar@/^/[r]^{Y_{\cat{Q}}(s)}\ar@/_/[r]_{Y_{\cat{Q}}(t)} & Y_{\cat{Q}}(1)}\in\cat{Q}$ are determined uniquely by $\vec{V}Y_{\cat{Q}}(s)(1)=(0,1)$ and $\vec{V}Y_{\cat{Q}}(t)(1)=(1,1)$, mapping to the tail and head of the single edge, respectively.
\end{prop}

Now, the box exponential and its universal property can be clearly stated and proven, complete with evaluation morphisms.  Please note the use of the right unitor and the Yoneda map in the source and target functions.

\begin{defn}[Box exponential]
Given quivers $Q_{1}$ and $Q_{2}$, define the quiver $\left[Q_{1},Q_{2}\right]_{B}$ by
\begin{enumerate}
\item $\vec{V}\left[Q_{1},Q_{2}\right]_{B}:=\cat{Q}\left(Q_{1},Q_{2}\right)$,
\item $\vec{E}\left[Q_{1},Q_{2}\right]_{B}:=\cat{Q}\left(Q_{1}\vec{\Box}\vec{E}^{\diamond}(\{1\}),Q_{2}\right)$,
\item $\sigma_{\left[Q_{1},Q_{2}\right]_{B}}(\psi):=\psi\circ\left(Q_{1}\vec{\Box}Y_{\cat{Q}}(s)\right)\circ\vec{r}^{-1}_{Q_{1}}$,
\item $\tau_{\left[Q_{1},Q_{2}\right]_{B}}(\psi):=\psi\circ\left(Q_{1}\vec{\Box}Y_{\cat{Q}}(t)\right)\circ\vec{r}^{-1}_{Q_{1}}$.
\end{enumerate}
Define the quiver homomorphism $\xymatrix{Q_{1}\vec{\Box}\left[Q_{1},Q_{2}\right]_{B}\ar[r]^(0.7){\mathrm{bev}_{Q_{2}}^{Q_{1}}} & Q_{2}}$ by
\begin{enumerate}
\item $\vec{V}\left(\mathrm{bev}_{Q_{2}}^{Q_{1}}\right)(v,\phi):=\vec{V}(\phi)(v)$,
\item $\vec{E}\left(\mathrm{bev}_{Q_{2}}^{Q_{1}}\right)(n,x,\psi):=\left\{\begin{array}{cc}
\vec{E}(\psi)(x)	&	n=1,\\
\vec{E}(\psi)(2,x,1)	&	n=2.\\
\end{array}\right.$
\end{enumerate}
\end{defn}

Observe that the edges are determined by the $\cat{Q}$-morphisms of a $Q_1$-prism, $\cat{Q}\left(Q_{1}\vec{\Box}\vec{E}^{\diamond}(\{1\}),Q_{2}\right)$.

\begin{ex}
Again, we are concerned with the objects produced by the exponential. Consider the quiver box exponential of a $2$-cycle to a $1$-edge. The vertex set is determined by maps from $\vec{P}_1$ to $\vec{C}_2$, which is uniquely determined by the image of the edge. The edge set is determined by maps from $\vec{P}_1\vec{\Box}\vec{E}^{\diamond}(\{1\}) = \vec{P}_1\vec{\Box}\vec{P}_1$ to $\vec{C}_2$, which is uniquely determined by the image of $(e,0)$.
\begin{figure}[H]
    \centering
    \includegraphics{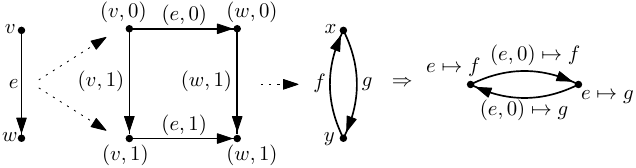}\\
    \caption{The quiver box exponential $[\vec{P}_1,\vec{C}_2]_B$ as determined by their maps.}
    \label{fig:QBoxExp}
\end{figure}

\end{ex}

\begin{thm}[Universal property]
Given a quiver homomorphism $\xymatrix{Q_{1}\vec{\Box}K\ar[r]^(0.6){\phi} & Q_{2}}$, there is a unique quiver homomorphism $\xymatrix{K\ar[r]^(0.3){\hat{\phi}} & \left[Q_{1},Q_{2}\right]_{B}}$ such that $\mathrm{bev}_{Q_{2}}^{Q_{1}}\circ\left(Q_{1}\vec{\Box}\hat{\phi}\right)=\phi$.
\end{thm}

\begin{proof}

For $v\in\vec{V}(K)$, define $\gamma_{v}:\{1\}\to\vec{V}(K)$ by $\gamma_{v}(1):=v$.  There is a unique $\xymatrix{\vec{V}^{\diamond}(\{1\})\ar[r]^(0.6){\hat{\gamma}_{v}} & K}\in\cat{Q}$ such that $\vec{V}\left(\hat{\gamma}_{v}\right)=\gamma_{v}$.  For $e\in\vec{E}(K)$, define $\delta_{e}:\{1\}\to\vec{E}(K)$ by $\delta_{e}(1):=e$.  There is a unique $\xymatrix{\vec{E}^{\diamond}(\{1\})\ar[r]^(0.6){\hat{\delta}_{e}} & K}\in\cat{Q}$ such that $\vec{E}\left(\hat{\delta}_{e}\right)=\delta_{e}$.  Define $\xymatrix{K\ar[r]^(0.35){\hat{\phi}} & \left[Q_{1},Q_{2}\right]_{B}}\in\cat{Q}$ by $\vec{V}\left(\hat{\phi}\right)(v):=\phi\circ\left(Q_{1}\vec{\Box}\hat{\gamma}_{v}\right)\circ\vec{r}_{Q_{1}}^{-1}$, and $\vec{E}\left(\hat{\phi}\right)(e):=\phi\circ\left(Q_{1}\vec{\Box}\hat{\delta}_{e}\right)$.
\end{proof}

\subsection{Box Product for Set System Hypergraphs}
\label{CatH}
Recall, that a set system hypergraph $H$ consists of a set of vertices $V(H)$, a set of edges $E(H)$, and an endpoint map $\epsilon_{H}$. The box product for set-system hypergraphs is defined analogously to its quiver counterpart with monoidal structure in Section \ref{StrMapsH}.

\begin{defn}[Box product]
Given set-system hypergraphs $G$ and $H$, define the set-system hypergraph $G\Box H$ by
\begin{enumerate}
\item $V\left(G\Box H\right):=V(G)\times V(H)$,
\item $E\left(G\Box H\right):=\left(\{1\}\times E(G)\times V(H)\right)\cup\left(\{2\}\times V(G)\times E(H)\right)$,
\item $\epsilon_{G\Box H}(n,x,y):=\left\{\begin{array}{cc}
\epsilon_{G}(x)\times\{y\},	&	n=1,\\
\{x\}\times\epsilon_{H}(y),	&	n=2.\\
\end{array}\right.$
\end{enumerate}
Given set-system hypergraph homomorphisms $\xymatrix{G_{1}\ar[r]^{\phi} & G_{2}}$ and $\xymatrix{H_{1}\ar[r]^{\psi} & H_{2}}$, define the set-system homomorphism $\xymatrix{G_{1}\Box H_{1}\ar[r]^{\phi\Box\psi} & G_{2}\Box H_{2}}$ by
\begin{enumerate}
\item $V\left(\phi\Box\psi\right)(v,w):=\left(V(\phi)(v),V(\psi)(w)\right)$,
\item $E\left(\phi\Box\psi\right)(n,x,y):=\left\{\begin{array}{cc}
\left(1,E(\phi)(x),V(\psi)(y)\right),	&	n=1,\\
\left(2,V(\phi)(x),E(\psi)(y)\right),	&	n=2.\\
\end{array}\right.$
\end{enumerate}
\end{defn}

\begin{ex}\label{h-box}
As one of the names of the box product is the ``Cartesian'' product, and the set system box product behaves exactly as expected.
\begin{figure}[H]
    \centering
    \includegraphics{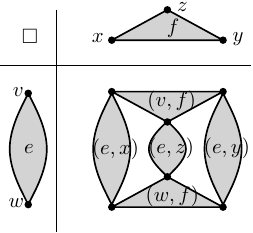}\\
    \caption{The set system box product of a $2$-edge and a $3$-edge.}
    \label{fig:HBoxProd}
\end{figure}
As \cite{IH1} discussed the cartesian monoidal structure of the category, we avoid the term ``Cartesian product'' to prevent confusion.
\end{ex}

The vertex functor $V$ for $\cat{H}$ admits a left adjoint, so the vertex set for the set-system box exponential is formed much like its quiver counterpart.
\begin{align*}
V[G,H]_{\beta}
&\cong\Set\left(\{1\},V[G,H]_{\beta}\right) \cong\cat{H}\left(V^{\diamond}(\{1\}),[G,H]_{\beta}\right) \\ &\cong\cat{H}\left(G\Box V^{\diamond}(\{1\}),H\right)
\cong\cat{H}\left(G,H\right).
\end{align*}

Unfortunately, the edge functor $E$ does not admit a left adjoint \cite[Lemma 2.217]{IH1}, so the edge set requires more careful consideration.  The counit $\beta\textrm{ev}_{H}^{G}$ of the exponential adjunction must be a set-system hypergraph homomorphism from $G\Box[G,H]_{\beta}$ to $H$, giving a map from $\{2\}\times V(G)\times E[G,H]_{\beta}$ to $E(H)$.  Thus, the edges of $[G,H]_{\beta}$ involve functions from $V(G)$ to $E(H)$.  Moreover, the homomorphism condition requires that the functions be colored by their endpoint set, giving the structure below.

\begin{defn}[Box exponential]\label{box-exponential-h}
Given set-system hypergraphs $G$ and $H$, define the hypergraph $[G,H]_{\beta}$ by
\begin{enumerate}
\item $V[G,H]_{\beta}:=\cat{H}(G,H)$ with evaluation map $V\left(\mathrm{\beta ev}_{H}^{G}\right):V\left(G\Box [G,H]_{\beta}\right)\to V(H)$ by $V\left(\mathrm{\beta ev}_{H}^{G}\right)(v,\phi):=V(\phi)(v)$,
\item $E[G,H]_{\beta}:=\left\{(A,g)\in\mathcal{P}V[G,H]_{\beta}\times\Set\left(V(G),E(H)\right):
\left(\epsilon_{H}\circ g\right)(v)=\right.$ \\ $\left.\mathcal{P}V\left(\mathrm{\beta ev}_{H}^{G}\right)\left(\{v\}\times A\right)
\forall v\in V(G)
\right\}$,
\item $\epsilon_{[G,H]_{\beta}}(A,g):=A$.
\end{enumerate}
Define the set-system hypergraph homomorphism $\xymatrix{G\Box[G,H]_{\beta}\ar[r]^(0.7){\mathrm{\beta ev}_{H}^{G}} & H}$ by
\begin{enumerate}
\item $V\left(\mathrm{\beta ev}_{H}^{G}\right)(v,\phi):=V(\phi)(v)$, $E\left(\mathrm{\beta ev}_{H}^{G}\right)(1,e,\phi):=E(\phi)(e)$,
\item $E\left(\mathrm{\beta ev}_{H}^{G}\right)\left(2,v,(A,g)\right):=g(v)$.
\end{enumerate}
\end{defn}

\begin{ex}
The set-system box exponential of a $2$-cycle to a $1$-edge is rather messy. While the vertex set consists of the standard $2^2$ vertices, the edge set contains $2^2$ functions colored by sets $A$ satisfying $\left\{V(\phi)(z):\forall\phi\in A\right\}=\{x,y\}$ for $z=v,w$. There are a total of eight $2$-edges, sixteen $3$-edges, and four $4$-edges, all in sets of four parallel edges. 
\begin{figure}[H]
    \centering
    \includegraphics{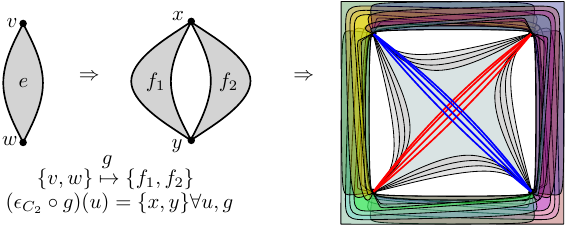}\\
    \caption{The set-system box exponential $[{P}_1,{C}_2]_{\beta}$ as determined by their maps.}
    \label{fig:HBoxExp}
\end{figure}
\end{ex}

\begin{thm}[Universal property]
Given a set-system hypergraph homomorphism $\xymatrix{G\Box K\ar[r]^(0.6){\phi} & H}$, there is a unique set-system hypergraph homomorphism $\xymatrix{K\ar[r]^(0.4){\hat{\phi}} & [G,H]_{\beta}}$ such that $\mathrm{\beta ev}_{H}^{G}\circ\left(G\Box\hat{\phi}\right)=\phi$.
\end{thm}

\begin{proof}

For $w\in V(K)$, define $\xymatrix{G\ar[r]^{V\left(\hat{\phi}\right)(w)} & H}\in\cat{H}$ by $V\left(V\left(\hat{\phi}\right)(w)\right)(v):=V(\phi)(v,w)$ and $E\left(V\left(\hat{\phi}\right)(w)\right)(e):=E(\phi)(1,e,w)$.  For $f\in E(K)$, define $g_{f}:V(G)\to E(H)$ by $g_{f}(v):=E(\phi)(2,v,f)$ and $A_{f}:=\left(\mathcal{P}V\left(\hat{\phi}\right)\circ\epsilon_{K}\right)(f)$.  Let $E\left(\hat{\phi}\right):E(K)\to E[G,H]_{\beta}$ by $E\left(\hat{\phi}\right)(f):=\left(A_{f},g_{f}\right)$ and $\hat{\phi}:=\left(E\left(\hat{\phi}\right),V\left(\hat{\phi}\right)\right)$.

\end{proof}

\subsection{Box Product for Set System Multigraphs}
\label{CatM}

As the category $\cat{M}$ of set-system multigraphs is a full subcategory of $\cat{H}$, $\cat{M}$ inherits the box product from $\cat{H}$. One can quickly check that the box product of two multigraphs is again a multigraph. Recall from \cite[Theorem 2.33]{IH1} that the inclusion functor $\xymatrix{\cat{M}\ar[r]^{N} & \cat{H}}$ which treats graphs as set systems with all edges size $2$, admits a right adjoint in the deletion functor $\xymatrix{\cat{H}\ar[r]^{\Del} & \cat{M}}$, which removes larger edges.  As $N$ has no affect on multigraphs or their morphisms, and $\Del$ only restricts the edge sets and maps, both become strict symmetric monoidal functors.  Moreover, $\cat{M}$ is closed by the calculation below.
\begin{align*}
\cat{M}\left(G\Box K,H\right) &= \cat{H}\left(G\Box K,H\right)
\cong\cat{H}\left(K,[G,H]_{\beta}\right) \\
&= \cat{H}\left(U(K),[G,H]_{\beta}\right)
\cong\cat{M}\left(K,\Del[G,H]_{\beta}\right).
\end{align*}

The underlying multigraph functor $\xymatrix{\cat{Q}\ar[r]^{U} & \cat{M}}$, which ``undirects'' directed edges, also does not have any substantial effect on morphisms.  As the monoidal structure for $\cat{Q}$ and $\cat{M}$ are nearly identical, similar calculations show that $U$ is another strict symmetric monoidal functor.  Moreover, in the case of simple graphs the box exponential for $\cat{M}$ matches that in \cite{DOCHTERMANN2009490,knauer-book}. The results are summarized below.

\begin{thm}[Inheritance of the box product]
For $G,H\in\ob(\cat{M})$, one has $G\Box H\in\ob(\cat{M})$.  Consequently, $\Box$ defines a closed symmetric monoidal product on $\cat{M}$.  Moreover, all of $N$, $\Del$, and $U$ are strict symmetric monoidal functors.
\end{thm}

By \cite[Theorem 2.37]{IH1}, $U$ admits a right adjoint $\xymatrix{\cat{M}\ar[r]^{\vec{D}} & \cat{Q}}$ determined by the associated digraph.  By \cite[p.\ 105]{lipman-hashimoto}, the strict monoidal structure of $U$ yields a lax monoidal structure for $\vec{D}$, but the structure maps are actually isomorphisms, giving the result below.

\begin{cor}[Symmetric monoidal functor $\vec{D}$]
The functor $\vec{D}$ is strong symmetric monoidal from $\left(\cat{M},\Box,V^{\diamond}(\{1\})\right)$ to $\left(\cat{Q},\vec{\Box},\vec{V}^{\diamond}(\{1\})\right)$.
\end{cor}

\begin{proof}

The counit of the $U$-$\vec{D}$ adjunction $\xymatrix{U\vec{D}(G)\ar[r]^(0.6){\theta_{G}} & G}\in\cat{M}$ is given by $V\left(\theta_{Q}\right)(v)=v$ and $E\left(\theta_{Q}\right)(e,x,y)=e$, while the unit $\xymatrix{Q\ar[r]^(0.4){\theta_{Q}^{\diamond}} & \vec{D}U(Q)}\in\cat{Q}$ is given by $\vec{V}\left(\theta_{Q}^{\diamond}\right)(v)=v$ and $\vec{E}\left(\theta_{Q}^{\diamond}\right)(e)=\left(e,\sigma_{Q}(e),\tau_{Q}(e)\right)$.  For $G,H\in\ob(\cat{M})$, the lax monoidal structure for $\vec{D}$ is given by $\psi_{G,H}:=\vec{D}\left(\theta_{G}\Box\theta_{H}\right)\circ\theta_{\vec{D}(G)\vec{\Box}\vec{D}(H)}^{\diamond}$ and $\psi_{\bullet}:=\vec{D}\left(id_{V^{\diamond}(\{1\})}\right)\circ\theta_{\vec{V}^{\diamond}(\{1\})}^{\diamond}$.  Routine calculations show that
\begin{itemize}
\item $\vec{V}\left(\psi_{G,H}\right)(v,w)=(v,w)$,
\item $\vec{E}\left(\psi_{G,H}\right)(1,(e,v,z),w)=((1,e,w),(v,w),(z,w))$,\\
$\vec{E}\left(\psi_{G,H}\right)(2,v,(f,w,u))=((2,v,f),(v,w),(v,u))$,
\item $\vec{V}\left(\psi_{\bullet}\right)(1)=1$,
$\vec{E}\left(\psi_{\bullet}\right)=id_{\emptyset}$.
\end{itemize}
Thus, both $\psi_{G,H}$ and $\psi_{\bullet}$ are isomorphisms.

\end{proof}

\begin{ex}
The box product on $\cat{M}$ is the canonical box product in Figure \ref{fig:PrismLadder}, while the box exponential is obtained from the $\cat{H}$ then applying $\Del$ to only leave the $2$-edges. Figure \ref{fig:HBoxExp} as a multigraph box exponential would consist only of the eight $2$-edges.
\end{ex}

\subsection{Box Product for Incidence Hypergraphs}
\label{CatR}
Recall, an incidence hypergraph $G$ consists of a set of vertices $\check{V}(G)$, a set of edges $\check{E}(G)$, a set of incidences $I(G)$, and two incidence maps $\varsigma_{G}:I(G)\to\check{V}(G)$, and $\omega_{G}:I(G)\to\check{E}(G)$. Taking inspiration from the quiver and set-system cases, a box product for incidence hypergraphs can be defined accordingly with the monoidal structure in Section \ref{StrMapsR}.

\begin{defn}[Box product]
Given incidence hypergraphs $G$ and $H$, define the incidence hypergraph $G\check{\Box}H$ by
\begin{enumerate}
\item $\check{V}\left(G\check{\Box}H\right):=\check{V}(G)\times\check{V}(H)$,
\item $\check{E}\left(G\check{\Box}H\right):=\left(\{1\}\times\check{E}(G)\times\check{V}(H)\right)\cup\left(\{2\}\times\check{V}(G)\times\check{E}(H)\right)$,
\item $I\left(G\check{\Box}H\right):=\left(\{1\}\times I(G)\times\check{V}(H)\right)\cup\left(\{2\}\times\check{V}(G)\times I(H)\right)$,
\item $\varsigma_{G\check{\Box}H}(n,x,y):=\left\{\begin{array}{cc}
\left(\varsigma_{G}(x),y\right),	&	n=1,\\
\left(x,\varsigma_{H}(y)\right),	&	n=2,\\
\end{array}\right.$
\item $\omega_{G\check{\Box}H}(n,x,y):=\left\{\begin{array}{cc}
\left(1,\omega_{G}(x),y\right),	&	n=1,\\
\left(2,x,\omega_{H}(y)\right),	&	n=2.\\
\end{array}\right.$
\end{enumerate}
Given incidence hypergraph homomorphisms $\xymatrix{G_{1}\ar[r]^{\phi} & G_{2}}$ and $\xymatrix{H_{1}\ar[r]^{\psi} & H_{2}}$, define the incidence hypergraph homomorphism $\xymatrix{G_{1}\check{\Box}H_{1}\ar[r]^{\phi\check{\Box}\psi} & G_{2}\check{\Box}H_{2}}$ by
\begin{enumerate}
\item $\check{V}\left(\phi\check{\Box}\psi\right)(v,w):=\left(\check{V}(\phi)(v),\check{V}(\psi)(w)\right)$,
\item $\check{E}\left(\phi\check{\Box}\psi\right)(n,x,y):=\left\{\begin{array}{cc}
\left(1,\check{E}(\phi)(x),\check{V}(\psi)(y)\right),	&	n=1,\\
\left(2,\check{V}(\phi)(x),\check{E}(\psi)(y)\right),	&	n=2,\\
\end{array}\right.$
\item $I\left(\phi\check{\Box}\psi\right)(n,x,y):=\left\{\begin{array}{cc}
\left(1,I(\phi)(x),\check{V}(\psi)(y)\right),	&	n=1,\\
\left(2,\check{V}(\phi)(x),I(\psi)(y)\right),	&	n=2.\\
\end{array}\right.$
\end{enumerate}
\end{defn}

\begin{ex}
By its construction, the box product for incidence hypergraphs agrees with the structure of the set system box product with the relevant incidences.
\begin{figure}[H]
    \centering
    \includegraphics{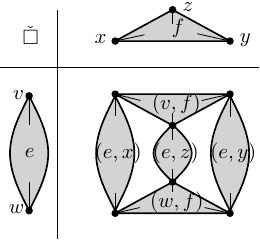}\\
    \caption{The incidence hypergraph box product of a $2$-edge and a $3$-edge.}
    \label{fig:RBoxProd}
\end{figure}
``Forgetting'' these incidences via $\mathscr{F}$ will return the appropriate set system box product. However, it was shown in \cite[Lemma 3.34]{IH1} that ``forgetting'' incidences is not functorial.  On the other hand, applying the incidence-forming functor from \cite[Lemma 3.32]{IH1} to Example \ref{h-box} yields this example precisely.  That is, the objects align for both $\mathcal{I}$ and $\mathscr{F}$, but the morphisms do not for $\mathscr{F}$.
\end{ex}

As with the quiver case, this monoidal product is closed, and the structure of the incidence box exponential is developed in parallel to its quiver counterpart.
\begin{align*}
\check{V}\left[G,H\right]_{V}
&\cong\Set\left(\{1\},\check{V}\left[G,H\right]_{V}\right)
\cong\cat{R}\left(\check{V}^{\diamond}\left(\{1\}\right),\left[G,H\right]_{V}\right) \\
&\cong\cat{R}\left(G\check{\Box}\check{V}^{\diamond}\left(\{1\}\right),H\right)
\cong\cat{R}\left(G,H\right), \\
\check{E}\left[G,H\right]_{V}
&\cong\Set\left(\{1\},\check{E}\left[G,H\right]_{V}\right)
\cong\cat{R}\left(\check{E}^{\diamond}\left(\{1\}\right),\left[G,H\right]_{V}\right)
\cong\cat{R}\left(G\check{\Box}\check{E}^{\diamond}\left(\{1\}\right),H\right), \\
I\left[G,H\right]_{V}
&\cong\Set\left(\{1\},I\left[G,H\right]_{V}\right)
\cong\cat{R}\left(I^{\diamond}\left(\{1\}\right),\left[G,H\right]_{V}\right)
\cong\cat{R}\left(G\check{\Box}I^{\diamond}\left(\{1\}\right),H\right).
\end{align*}

However, a peculiar change occurs for the edge set.  Direct calculation shows that
\begin{align*}
G\check{\Box}\check{E}^{\diamond}(\{1\})=\check{E}^{\diamond}\left(\{2\}\times\check{V}(G)\times\{1\}\right).
\end{align*}
Therefore, the edge set resolves to be far simpler, and familiar,
\begin{align*}
\check{E}\left[G,H\right]_{V}
&\cong\cat{R}\left(\check{E}^{\diamond}\left(\{2\}\times\check{V}(G)\times\{1\}\right),H\right)
\cong\Set\left(\{2\}\times\check{V}(G)\times\{1\},\check{E}(H)\right) \\
&\cong\Set\left(\check{V}(G),\check{E}(H)\right).
\end{align*}

Much like the set-system case, the edges of the incidence box exponential involve functions from the vertices to the edges.  However, there is no need to color the functions by the endpoint set, streamlining the construction.  As in the quiver case, the port and attachment functions are determined by the Yoneda embedding.  Again, the characterization of this functor follows from direct calculation.

\begin{prop}[Yoneda functor]
Let $Y_{\cat{R}}:\cat{D}^{\mathrm{op}}\to\cat{R}$ be the Yoneda embedding.  Then, $Y_{\cat{R}}(0)\cong\check{V}^{\diamond}(\{1\})$, $Y_{\cat{R}}(1)\cong\check{E}^{\diamond}(\{1\})$, and $Y_{\cat{R}}(2)\cong I^{\diamond}(\{1\})$.  Moreover, $\xymatrix{Y_{\cat{R}}(0)\ar[r]^{Y_{\cat{R}}(y)} & Y_{\cat{R}}(2) & Y_{\cat{R}}(1)\ar[l]_{Y_{\cat{R}}(z)}}\in\cat{R}$ are determined uniquely by $\check{V}Y_{\cat{R}}(y)(1)=1$ and $\check{E}Y_{\cat{R}}(z)(1)=1$, mapping to the only vertex and edge, respectively.
\end{prop}

With all components ready, the box exponential and its universal property can be stated and proven.    Again, observe the use of the right unitor and Yoneda map in the port function.

\begin{defn}[Box exponential]
Given incidence hypergraphs $G$ and $H$, define the hypergraph $[G,H]_{V}$ by
\begin{enumerate}
\item $\check{V}[G,H]_{V}:=\cat{R}(G,H)$,
\item $\check{E}[G,H]_{V}:=\Set\left(\check{V}(G),\check{E}(H)\right)$,
\item $I[G,H]_{V}:=\cat{R}\left(G\check{\Box}I^{\diamond}(\{1\}),H\right)$,
\item $\varsigma_{[G,H]_{V}}(\psi):=\psi\circ\left(G\check{\Box}Y_{\cat{R}}(y)\right)\circ\check{r}^{-1}_{G}$,
\item $\left(\omega_{[G,H]_{V}}(\psi)\right)(v):=\check{E}(\psi)(2,v,1)$.
\end{enumerate}
Define the incidence hypergraph homomorphism $\xymatrix{G\check{\Box}[G,H]_{V}\ar[r]^(0.7){\mathrm{vev}_{H}^{G}} & H}$ by
\begin{enumerate}
\item $\check{V}\left(\mathrm{vev}_{H}^{G}\right)(v,\phi):=\check{V}(\phi)(v)$,
\item $\check{E}\left(\mathrm{vev}_{H}^{G}\right)(n,x,\psi):=\left\{\begin{array}{cc}
\check{E}(\psi)(x),	&	n=1,\\
\psi(x),	&	n=2,\\
\end{array}\right.$
\item $I\left(\mathrm{vev}_{H}^{G}\right)(n,x,\varphi):=\left\{\begin{array}{cc}
I(\varphi)(x)	&	n=1,\\
I(\varphi)(2,x,1)	&	n=2.\\
\end{array}\right.$
\end{enumerate}
\end{defn}

Observe that the edge set does not consist of $\cat{R}$-morphisms, while the incidences are determined by the $\cat{R}$-morphisms of a incidence-$G$-prism, $\cat{R}\left(G\check{\Box}{I}^{\diamond}(\{1\}),H\right)$. The asymmetry for the edge set is remedied in Section \ref{sec:LapProd}.

\begin{ex}
Consider the incidence hypergraph box exponential of $\check{P}_1$, the path of length $1$, to the terminal object $I^{\diamond}(\{1\})$, the single incidence $1$-edge. The vertex set of $[I^{\diamond}(\{1\}),{\check{P}}_1]_{V}$ are the $\cat{R}$-morphisms from $I^{\diamond}(\{1\}) \rightarrow {\check{P}}_1$, which are determined by $i \mapsto j$ or $i \mapsto k$ in Figure \ref{fig:RBoxExp}. The edges of $[I^{\diamond}(\{1\}),{\check{P}}_1]_{V}$ are the $\Set$-morphisms that map the vertices of $I^{\diamond}(\{1\})$ to the edges of ${\check{P}}_1$. There is only one such map $v \mapsto f$.
\begin{figure}[H]
    \centering
    \includegraphics{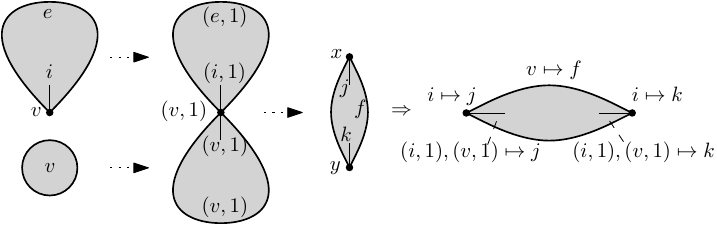}\\
    \caption{The $\cat{R}$ box exponential $[I^{\diamond}(\{1\}),{\check{P}}_1]_{V}$ as determined by their maps.}
    \label{fig:RBoxExp}
\end{figure}
The incidences are calculated in Figure \ref{fig:RBoxExp} via the $\cat{R}$-morphism through-maps from 
\begin{align*}
I^{\diamond}(\{1\}) \check{\square}I^{\diamond}(\{1\} ) \rightarrow {\check{P}}_1,
\end{align*}
which is uniquely determined by the image of $(v,1)$.
\end{ex}

\begin{thm}[Universal property]
Given an incidence hypergraph homomorphism $\xymatrix{G\check{\Box}K\ar[r]^(0.6){\phi} & H}$, there is a unique incidence hypergraph homomorphism $\xymatrix{K\ar[r]^(0.4){\hat{\phi}} & [G,H]_{V}}$ such that $\mathrm{vev}_{H}^{G}\circ\left(G\check{\Box}\hat{\phi}\right)=\phi$.
\end{thm}

\begin{proof}

For $v\in\check{V}(K)$, define $\gamma_{v}:\{1\}\to\check{V}(K)$ by $\gamma_{v}(1):=v$.  There is a unique $\xymatrix{\check{V}^{\diamond}(\{1\})\ar[r]^(0.6){\hat{\gamma}_{v}} & K}\in\cat{R}$ such that $\check{V}\left(\hat{\gamma}_{v}\right)=\gamma_{v}.$
For $i\in I(K)$, define $\delta_{i}:\{1\}\to I(K)$ by $\delta_{i}(1):=i$.  There is a unique $\xymatrix{I^{\diamond}(\{1\})\ar[r]^(0.6){\hat{\delta}_{i}} & K}\in\cat{R}$ such that $I\left(\hat{\delta}_{i}\right)=\delta_{i}$.  Define $\xymatrix{K\ar[r]^(0.35){\hat{\phi}} & \left[G,H\right]_{V}}\in\cat{R}$ by
\begin{itemize}
\item $\check{V}\left(\hat{\phi}\right)(v):=\phi\circ\left(G\check{\Box}\hat{\gamma}_{v}\right)\circ\check{r}_{G}^{-1}$,
\item $\left(\check{E}\left(\hat{\phi}\right)(e)\right)(w):=\check{E}(\phi)(2,w,e)$,
\item $I\left(\hat{\phi}\right)(i):=\phi\circ\left(G\check{\Box}\hat{\delta}_{i}\right)$.
\end{itemize}

\end{proof}

Recall from \cite[Lemma 3.32]{IH1} that there is a natural incidence-forming functor $\xymatrix{\cat{H}\ar[r]^{\mathcal{I}} & \cat{R}}$, which was sadly neither continuous nor cocontinuous.  On the other hand, as the box product for incidence hypergraphs was based on the box product for set-system hypergraphs, $\mathcal{I}$ is a strong symmetric monoidal functor when using the respective box products.  The structure maps for $\mathcal{I}$ are defined below, and the verification is routine.

\begin{defn}[Monoidal structure for $\mathcal{I}$]
Given $G,H\in\ob(\cat{H})$, define $\xymatrix{\mathcal{I}(G)\check{\Box}\mathcal{I}(H)\ar[r]^{\Phi_{G,H}} & \mathcal{I}(G\Box H)}\in\cat{R}$ by
\begin{enumerate}
\item $\check{V}\left(\Phi_{G,H}\right)(v,w):=(v,w)$, $\check{E}\left(\Phi_{G,H}\right)(n,x,y):=(n,x,y)$,
\item $I\left(\Phi_{G,H}\right)(1,(v,e),w):=((v,w),(1,e,w))$,
\item $I\left(\Phi_{G,H}\right)(2,v,(w,f)):=((v,w),(2,v,f))$.
\end{enumerate}
\end{defn}

\section{Laplacian Product}
\label{sec:LapProd}

The spirit of this new product can be seen as an adaptation of the multiplication of complex numbers.  If one considers vertices as the ``real part'' of a graph, and edges as the ``imaginary part'', consider the multiplication below,
\begin{align*}
\left(v_{1}+\imath e_{1}\right)\left(v_{2}+\imath e_{2}\right)
=\left(v_{1}v_{2}-e_{1}e_{2}\right)+\imath\left(e_{1}v_{2}+v_{1}e_{2}\right).
\end{align*}
Observe that the components of the traditional box product have arisen naturally:  vertices $V_{1}\times V_{2}$, edges $E_{1}\times V_{2}$, and edges $V_{1}\times E_{2}$. However, a new set has arisen:  vertices $E_{1}\times E_{2}$. The inclusion of this new set of vertices builds a new box product, called the ``Laplacian product'' for reasons made clear in this section. In this section we provide a categorical formulation of incidence duality and generalize the box product to a dually-closed product that: (1) has a simple interpretation via bipartite graphs; (2) combinatorially treats vertices and edges as real and imaginary parts of a hypergraph, respectively; (3) has an exponential where all parts are homomorphisms in $\cat{R}$; and (4) the evaluation of this new exponential at paths determines the combinatorial Laplacian for powers of the introverted/extroverted oriented hypergraph (signless Laplacian) and its dual from \cite{OHSachs,IH2,AH1,OHMTT}.
\subsection{Incidence Duality and the Laplacian Product}

Notice that the finite category $\cat{D}$ in defining incidence hypergraphs (Appendix subsection \ref{IHdefn}, and Table \ref{t:library1}) has a symmetry about object $2$, and there is an obvious functor $\Sigma$ swapping the objects $0$ and $1$ (vertices and edges), and the morphisms $y$ and $z$ (port and attachment).  Composing an incidence hypergraph $G$ with $\Sigma$ gives rise to \emph{incidence duality} by reversing the roles of vertices and edges.  As $\Sigma$ is clearly its own inverse, hence, incidence duality is self-inverting.  These results are summarized below. Please note that both the logical functor $\Upsilon$ and the incidence-dual functor $\Box^{\#}$ arise naturally as composition functors.  This shared representation pattern raises the question of what other compositions might have significance.

\begin{defn}[Incidence duality]
Let $\Sigma:\cat{D}\to\cat{D}$ be the functor given by $y\mapsto z$ and $z\mapsto y$.  Define $(\cdot)^{\#}:=(\cdot)\Sigma$, the functor from $\cat{R}$ to itself determined by composing on the right by $\Sigma$.
\end{defn}

\begin{lem}[Action of $(\cdot)^{\#}$]
Given $\xymatrix{G\ar[r]^{\phi} & H}\in\cat{R}$, then
\begin{enumerate}
\item $G^{\#}=\left(\check{V}(G),\check{E}(G),I(G),\varsigma_{G},\omega_{G}\right)^{\#}=\left(\check{E}(G),\check{V}(G),I(G),\omega_{G},\varsigma_{G}\right)$,
\item $\phi^{\#}=\left(\check{V}(\phi),\check{E}(\phi),I(\phi)\right)^{\#}=\left(\check{E}(\phi),\check{V}(\phi),I(\phi)\right)$.
\end{enumerate}
\end{lem}

\begin{thm}[Properties of $(\cdot)^{\#}$]\label{incidence-dual}
The functor $(\cdot)^{\#}$ is self-inverting.  Moreover, the following functorial equalities hold:  $\check{V}\left((\cdot)^{\#}\right)=\check{E}(\cdot)$, $\check{E}\left((\cdot)^{\#}\right)=\check{V}(\cdot)$, $I\left((\cdot)^{\#}\right)=I(\cdot)$.
\end{thm}

\begin{ex}
The incidence dual of a path is shown in Figure \ref{fig:IncDual}. Here, the incidence dual of ${\check{P}}_1$ is also a path of length $1$ starting and ending at an edge.
\begin{figure}[H]
    \centering
    \includegraphics{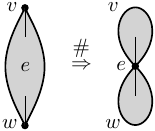}\\
    \caption{The incidence dual of ${\check{P}}_1$.}
    \label{fig:IncDual}
\end{figure}
\end{ex}

Including the ``missing'' vertices from $\check{\square}$ (compare the vertex set below to the $\cat{R}$-box-product) and adding the appropriate incidences, we obtain the following duality-closed version of box product.    As with the other products discussed thus far, the Laplacian product is a symmetric monoidal product with the structure maps in Section \ref{StrMapsL}.

\begin{defn}[Laplacian product]
Given incidence hypergraphs $G$ and $H$, define the hypergraph $G\blacksquare H$ by
\begin{enumerate}
\item $\check{V}\left(G\blacksquare H\right):=\left(\{1\}\times\check{V}(G)\times\check{V}(H)\right)\cup\left(\{4\}\times\check{E}(G)\times\check{E}(H)\right)$,
\item $\check{E}\left(G\blacksquare H\right):=\left(\{2\}\times\check{E}(G)\times\check{V}(H)\right)\cup\left(\{3\}\times\check{V}(G)\times\check{E}(H)\right)$,
\item $I\left(G\blacksquare H\right):=\left(\{1\}\times I(G)\times\check{V}(H)\right)\cup\left(\{2\}\times I(G)\times\check{E}(H)\right)\cup\left(\{3\}\times\check{E}(G)\times I(H)\right)\cup\left(\{4\}\times\check{V}(G)\times I(H)\right)$,
\item $\varsigma_{G\blacksquare H}(n,x,y):=\left\{\begin{array}{cc}
\left(1,\varsigma_{G}(x),y\right),	&	n=1,\\
\left(4,\omega_{G}(x),y\right),	&	n=2,\\
\left(4,x,\omega_{H}(y)\right),	&	n=3,\\
\left(1,x,\varsigma_{H}(y)\right),	&	n=4,\\
\end{array}\right.$
\item $\omega_{G\blacksquare H}(n,x,y):=\left\{\begin{array}{cc}
\left(2,\omega_{G}(x),y\right),	&	n=1,\\
\left(3,\varsigma_{G}(x),y\right),	&	n=2,\\
\left(2,x,\varsigma_{H}(y)\right),	&	n=3,\\
\left(3,x,\omega_{H}(y)\right),	&	n=4.\\
\end{array}\right.$
\end{enumerate}
Given incidence hypergraph homomorphisms $\xymatrix{G_{1}\ar[r]^{\phi} & G_{2}}$ and $\xymatrix{H_{1}\ar[r]^{\psi} & H_{2}}$, define the incidence hypergraph homomorphism $\xymatrix{G_{1}\blacksquare H_{1}\ar[r]^{\phi\blacksquare\psi} & G_{2}\blacksquare H_{2}}$ by
\begin{enumerate}
\item $\check{V}\left(\phi\blacksquare\psi\right)(n,x,y):=\left\{\begin{array}{cc}
\left(1,\check{V}(\phi)(x),\check{V}(\psi)(y)\right),	&	n=1,\\
\left(4,\check{E}(\phi)(x),\check{E}(\psi)(y)\right),	&	n=4,\\
\end{array}\right.$
\item $\check{E}\left(\phi\blacksquare\psi\right)(n,x,y):=\left\{\begin{array}{cc}
\left(2,\check{E}(\phi)(x),\check{V}(\psi)(y)\right),	&	n=2,\\
\left(3,\check{V}(\phi)(x),\check{E}(\psi)(y)\right),	&	n=3,\\
\end{array}\right.$
\item $I\left(\phi\blacksquare\psi\right)(n,x,y):=\left\{\begin{array}{cc}
\left(1,I(\phi)(x),\check{V}(\psi)(y)\right),	&	n=1,\\
\left(2,I(\phi)(x),\check{E}(\psi)(y)\right),	&	n=2,\\
\left(3,\check{E}(\phi)(x),I(\psi)(y)\right),	&	n=3,\\
\left(4,\check{V}(\phi)(x),I(\psi)(y)\right),	&	n=4.\\
\end{array}\right.$
\end{enumerate}
\end{defn}

The next example demonstrates the box-like nature of the Laplacian product over the incidence structure. Moreover, the product of edge-pairs are vertices, effectively treating edges as the ``imaginary part'' of an incidence hypergraph.

\begin{ex}
Consider the product of two single-incidence $1$-edge generators $\check{P}_{1/2} \cong I^{\diamond}\left(\{1\}\right)$.
By construction, the objects being multiplied are replaced with its incidence-dual as it traverses each incidence.  The two dual copies of the single incidence are dashed-circled (right), while the single incidence inducing the duality appear on the dotted-line.
\begin{figure}[H]
    \centering
    \includegraphics{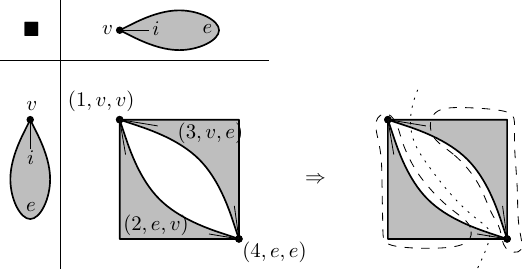}\\
    \caption{The Laplacian product of $\check{P}_{1} \blacksquare \check{P}_{1}$.}
    \label{fig:RLaplacianProd1}
\end{figure}
Here we see that the Laplacian product forms ``boxes'' along the incidence structure where products of edge-pairs are vertices.
\end{ex}

The product of an incidence hypergraph with the single incidence hypergraph $G \blacksquare \check{P}_{1/2} \cong G \blacksquare I^{\diamond}\left(\{1\}\right)$ is called the \emph{incidence-prism of $G$}. Compared to the standard prism $G \square {P}_{1}$ for graphs, the incidence-prism produces a copy of $G$ and its dual $G^{\#}$ linked via ``rung'' incidences. This is shown in the next example with the creation of an incidence-ladder. 

\begin{ex}
Now consider the Laplacian product of a $2$-edge with a $1$-edge, $\check{P}_{1} \blacksquare \check{P}_{1/2}$. The dual copies of ${\check{P}}_1$ and ${\check{P}}^{\#}_1$ from Figure \ref{fig:IncDual} appear in the Laplacian product in the dashed-circles in Figure \ref{fig:RLaplacianProd2} with the incidence $j$ tying them together. This produces an incidence-ladder in which the rungs are incidences.

\begin{figure}[H]
    \centering
    \includegraphics{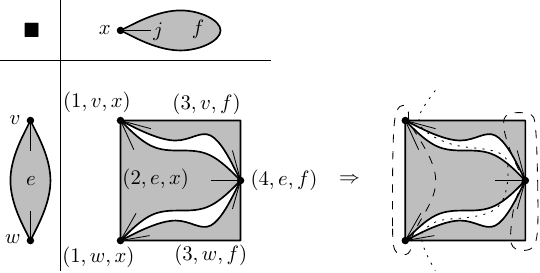}\\
    \caption{The Laplacian product of a $2$-edge with a $1$-edge.}
    \label{fig:RLaplacianProd2}
\end{figure}
Additionally, there are $3$ copies of the $1$-edge horizontally through the product, dualizing every incidence step. 
\end{ex}

As in the other box products, the single isolated vertex $\check{V}^{\diamond}(\{1\})$ is the unit object.  Since $\left(\check{V}^{\diamond}(\{1\})\right)^{\#}=\check{E}^{\diamond}(\{1\})$, the single loose edge, the latter has a similar action.  Instead of recovering the original object, the Laplacian product with $\check{E}^{\diamond}(\{1\})$ creates the incidence dual.  This action is implemented by ``anti-unitor'' natural isomorphisms defined below.

\begin{defn}[Anti-unitors]
For $G\in\ob(\cat{R})$, define $\xymatrix{G\blacksquare\check{E}^{\diamond}(\{1\})\ar[r]^(0.7){\hat{\rho}_{G}} & G^{\#} & \check{E}^{\diamond}(\{1\})\blacksquare G\ar[l]_(0.7){\hat{\lambda}_{G}}}\in\cat{R}$ by
\begin{enumerate}
\item $\check{V}\left(\hat{\rho}_{G}\right)(4,e,1):=e$, $\check{E}\left(\hat{\rho}_{G}\right)(3,v,1):=v$, $I\left(\hat{\rho}_{G}\right)(2,i,1):=i$,
\item $\check{V}\left(\hat{\lambda}_{G}\right)(4,1,e):=e$, $\check{E}\left(\hat{\lambda}_{G}\right)(2,1,v):=v$, $I\left(\hat{\lambda}_{G}\right)(3,1,i):=i$.
\end{enumerate}
\end{defn}

As with the unitors of the monoidal product, the anti-unitors entangle nicely with the commutator.  Like the monoidal structure, the proof is tedious, but routine.

\begin{lem}[Anti-unitors \& commutator]\label{antiunitor}
For $G\in\ob(\cat{R})$, the following diagram commutes.
\begin{align*}
\xymatrix{
G\blacksquare\check{E}^{\diamond}(\{1\})\ar[rr]^{\check{\gamma}_{G,\check{E}^{\diamond}(\{1\})}}\ar[dr]_{\hat{\rho}_{G}}	&	&	\check{E}^{\diamond}(\{1\})\blacksquare G\ar[dl]^{\hat{\lambda}_{G}}\\
&	G^{\#}\\
}
\end{align*}
\end{lem}

Combining the triangle from Lemma \ref{antiunitor} with the associator-commutator hexagon from \cite[Def. II.6.1.2]{borceux} yields the ``Triforce of Duality'' in Figure \ref{triforce}.  Thus, the incidence dual acting on a Laplacian product can be migrated to either coordinate of the product as stated in the theorem below. Recall, incidence duality acts as a hypergraphic replacement of line graphs, and acts as transposition on incidence matrices \cite{AH1}.

\begin{thm}[Duality \& Laplacian product]\label{dual-laplace}
For $G,H\in\ob(\cat{R})$, one has the following natural isomorphisms from Figure \ref{triforce}.
\begin{align*}
\left(G\blacksquare H\right)^{\#}
\cong G^{\#}\blacksquare H
\cong G\blacksquare H^{\#}
\end{align*}
\end{thm}

\begin{figure}[H]
\begin{align*}
\xymatrix{
&	&	&	G^{\#}\blacksquare H\\
&	&	\left(\check{E}^{\diamond}(\{1\})\blacksquare G\right)\blacksquare H\ar[dl]_{\check{\alpha}_{\check{E}^{\diamond}(\{1\}),G,H}}\ar[dr]_{\check{\gamma}_{\check{E}^{\diamond}(\{1\}),G}\blacksquare id_{H}}\ar[ur]^{\hat{\lambda}_{G}\blacksquare id_{H}}\\
&	\check{E}^{\diamond}(\{1\})\blacksquare\left(G\blacksquare H\right)\ar[dd]^{\check{\gamma}_{\check{E}^{\diamond}(\{1\}),G\blacksquare H}}\ar[dl]_{\hat{\lambda}_{G\blacksquare H}}	&	&	\left(G\blacksquare\check{E}^{\diamond}(\{1\})\right)\blacksquare H\ar[dd]^{\check{\alpha}_{G,\check{E}^{\diamond}(\{1\}),H}}\ar[uu]_{\hat{\rho}_{G}\blacksquare id_{H}}\\
\left(G\blacksquare H\right)^{\#}\\
&	\left(G\blacksquare H\right)\blacksquare\check{E}^{\diamond}(\{1\})\ar[dr]_{\check{\alpha}_{G,H,\check{E}^{\diamond}(\{1\})}}\ar[ul]^{\hat{\rho}_{G\blacksquare H}}	&	&	G\blacksquare\left(\check{E}^{\diamond}(\{1\})\blacksquare H\right)\ar[dl]^{id_{G}\blacksquare \check{\gamma}_{\check{E}^{\diamond}(\{1\}),H}}\ar[dd]^{id_{G}\blacksquare\hat{\lambda}_{H}}\\
&	&	G\blacksquare\left(H\blacksquare\check{E}^{\diamond}(\{1\})\right)\ar[dr]_{id_{G}\blacksquare\hat{\rho}_{H}}\\
&	&	&	G\blacksquare H^{\#}\\
}
\end{align*}
\caption{Triforce of Duality}
\label{triforce}
\end{figure}

Thus, dualizing a single hypergraph induces duality on the entire product.

\subsection{Laplacian Exponential}

Like all the previous cases, the Laplacian product has an associated exponential bracket, but its construction is far more symmetric than its predecessors.  This is due to the anti-unitor isomorphisms above.
\begin{align*}
\check{V}\left[G,H\right]_{L}
&\cong\Set\left(\{1\},\check{V}\left[G,H\right]_{L}\right)
\cong\cat{R}\left(\check{V}^{\diamond}\left(\{1\}\right),\left[G,H\right]_{L}\right) \\
&\cong\cat{R}\left(G\blacksquare\check{V}^{\diamond}\left(\{1\}\right),H\right)
\cong\cat{R}\left(G,H\right), \\
\check{E}\left[G,H\right]_{L}
&\cong\Set\left(\{1\},\check{E}\left[G,H\right]_{L}\right)
\cong\cat{R}\left(\check{E}^{\diamond}\left(\{1\}\right),\left[G,H\right]_{L}\right) \\
&\cong\cat{R}\left(G\blacksquare\check{E}^{\diamond}\left(\{1\}\right),H\right)
\cong\cat{R}\left(G^{\#},H\right), \\
I\left[G,H\right]_{L}
&\cong\Set\left(\{1\},I\left[G,H\right]_{L}\right)
\cong\cat{R}\left(I^{\diamond}\left(\{1\}\right),\left[G,H\right]_{L}\right)
\cong\cat{R}\left(G\blacksquare I^{\diamond}\left(\{1\}\right),H\right).
\end{align*}

Once more the Yoneda embedding provides the port and attachment functions, giving the construction below.  Moreover, right unitor and anti-unitor appear in the port and attachment functions, respectively with the Yoneda map.

\begin{defn}[Laplacian exponential]
Given incidence hypergraphs $G$ and $H$, define the hypergraph $[G,H]_{L}$ by
\begin{enumerate}
\item $\check{V}[G,H]_{L}:=\cat{R}(G,H)$,
\item $\check{E}[G,H]_{L}:=\cat{R}\left(G^{\#},H\right)$,
\item $I[G,H]_{L}:=\cat{R}\left(G\blacksquare I^{\diamond}(\{1\}),H\right)$,
\item $\varsigma_{[G,H]_{L}}(\psi):=\psi\circ\left(G\blacksquare Y_{\cat{R}}(y)\right)\circ\check{\rho}^{-1}_{G}$,
\item $\omega_{[G,H]_{L}}(\psi):=\psi\circ\left(G\blacksquare Y_{\cat{R}}(z)\right)\circ\hat{\rho}^{-1}_{G}$.
\end{enumerate}
Define the incidence hypergraph homomorphism $\xymatrix{G\blacksquare[G,H]_{L}\ar[r]^(0.7){\mathrm{cev}_{H}^{G}} & H}$ by
\begin{enumerate}
\item $\check{V}\left(\mathrm{cev}_{H}^{G}\right)(n,x,\phi):=\check{V}(\phi)(x)$,
\item $\check{E}\left(\mathrm{cev}_{H}^{G}\right)(n,x,\phi):=\check{E}(\phi)(x)$,
\item $I\left(\mathrm{cev}_{H}^{G}\right)(n,x,\psi):=\left\{\begin{array}{cc}
I(\psi)(x)	&	n=1,2,\\
I(\psi)(3,x,1)	&	n=3,\\
I(\psi)(4,x,1)	&	n=4.\\
\end{array}\right.$
\end{enumerate}
\label{def:LapExp}
\end{defn}

The edge set finally consists of $\cat{R}$-morphisms, while the incidences are determined by the $\cat{R}$-morphisms of the incidence-prism of $G$, $\cat{R}\left(G\blacksquare\check{P}_{1/2},H\right)$.

\begin{ex}
Consider the incidence hypergraph Laplacian exponential of $\check{P}_1$, the path of length $1$, to the terminal object $\check{P}_{1/2} \cong I^{\diamond}\left(\{1\}\right)$, the single incidence $1$-edge. The vertex set is identical to the  box exponential in $\cat{R}$. The edges of $[\check{P}_{1/2},{\check{P}}_1]_{L}$ are now the $\cat{R}$-morphisms from the dual (effectively addressing the set-crossing issue). Since $\check{P}_{1/2}$ is isomorphic to its dual the edges are calculated identically as the vertices. 
\begin{figure}[H]
    \centering
    \includegraphics{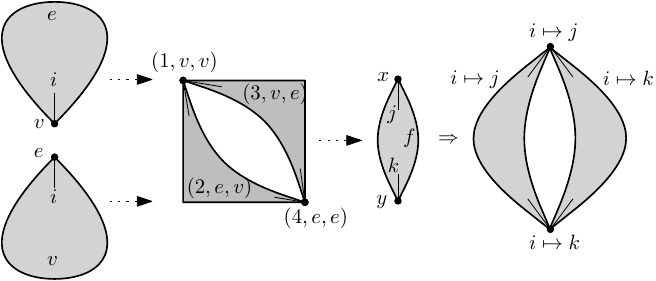}\\
    \caption{The Laplacian exponential $[\check{P}_{1/2},{\check{P}}_1]_{L}$ as determined by their maps.}
    \label{fig:RLaplacianExp}
\end{figure}
The incidences are calculated in Figure \ref{fig:RLaplacianExp} via the $\cat{R}$-morphism through-maps from $I^{\diamond}(\{1\}) {\blacksquare}I^{\diamond}(\{1\} )$ to ${\check{P}}_1$.  This Laplacian product was previously calculated in Figure \ref{fig:RLaplacianProd1}.
\end{ex}

\begin{thm}[Universal property]
Given an incidence hypergraph homomorphism $\xymatrix{G\blacksquare K\ar[r]^(0.6){\phi} & H}$, there is a unique incidence hypergraph homomorphism $\xymatrix{K\ar[r]^(0.4){\hat{\phi}} & [G,H]_{L}}$ such that $\mathrm{cev}_{H}^{G}\circ\left(G\blacksquare\hat{\phi}\right)=\phi$.
\end{thm}

\begin{proof}

For $v\in\check{V}(K)$, define $\gamma_{v}:\{1\}\to\check{V}(K)$ by $\gamma_{v}(1):=v$.  There is a unique $\xymatrix{\check{V}^{\diamond}(\{1\})\ar[r]^(0.6){\hat{\gamma}_{v}} & K}\in\cat{R}$ such that $\check{V}\left(\hat{\gamma}_{v}\right)=\gamma_{v}$.  For $e\in\check{E}(K)$, define $\delta_{e}:\{1\}\to\check{E}(K)$ by $\delta_{e}(1):=e$.  There is a unique $\xymatrix{\check{E}^{\diamond}(\{1\})\ar[r]^(0.6){\hat{\delta}_{e}} & K}\in\cat{R}$ such that $\check{E}\left(\hat{\delta}_{e}\right)=\delta_{e}$.  For $i\in I(K)$, define $\theta_{i}:\{1\}\to I(K)$ by $\theta_{i}(1):=i$.  There is a unique $\xymatrix{I^{\diamond}(\{1\})\ar[r]^(0.6){\hat{\theta}_{i}} & K}\in\cat{R}$ such that $I\left(\hat{\theta}_{i}\right)=\theta_{i}$.  Define $\xymatrix{K\ar[r]^(0.35){\hat{\phi}} & \left[G,H\right]_{L}}\in\cat{R}$ by
\begin{itemize}
\item $\check{V}\left(\hat{\phi}\right)(v):=\phi\circ\left(G\blacksquare\hat{\gamma}_{v}\right)\circ\check{\rho}_{G}^{-1}$,
\item $\check{E}\left(\hat{\phi}\right)(e):=\phi\circ\left(G\blacksquare\hat{\delta}_{e}\right)\circ\hat{\rho}_{G}^{-1}$,
\item $I\left(\hat{\phi}\right)(i):=\phi\circ\left(G\blacksquare\hat{\theta}_{i}\right)$.
\end{itemize}

\end{proof}

Due to the Triforce of Duality, the Laplacian exponential inherits the same duality relationships as the Laplacian product.

\begin{cor}[Duality \& Laplacian Exponential]
For $G,H\in\ob(\cat{R})$, one has the following natural isomorphisms:  $\left[G,H\right]_{L}^{\#}\cong\left[G,H^{\#}\right]_{L}\cong\left[G^{\#},H\right]_{L}$.
\label{c:LapExpDual}
\end{cor}

\begin{proof}
By Theorems \ref{incidence-dual} and \ref{dual-laplace}, the following natural isomorphisms result for all $K\in\ob(\cat{R})$:
\begin{itemize}
\item $\cat{R}\left(\left(G\blacksquare K\right)^{\#},H\right)
\cong\cat{R}\left(G\blacksquare K,H^{\#}\right)
\cong\cat{R}\left(K,\left[G,H^{\#}\right]_{L}\right),
$
\item $\cat{R}\left(\left(G\blacksquare K\right)^{\#},H\right)
\cong\cat{R}\left(G^{\#}\blacksquare K,H\right)
\cong\cat{R}\left(K,\left[G^{\#},H\right]_{L}\right),
$
\item $\cat{R}\left(\left(G\blacksquare K\right)^{\#},H\right)
\cong\cat{R}\left(G\blacksquare K^{\#},H\right)
\cong\cat{R}\left(K^{\#},\left[G,H\right]_{L}\right)
\cong\cat{R}\left(K,\left[G,H\right]_{L}^{\#}\right).
$
\end{itemize}
\end{proof}

\section{Graphical Interpretation and Examples}
\label{sec:upsilon}

The left adjoint to the logical functor $\xymatrix{\cat{Q}\ar[r]^{\Upsilon} & \cat{R}}$ introduced in \cite{IH1} when composed with the undirecting functor $U$ produces the bipartite representation graph as depicted in Figure \ref{fig:GandBipartite}. We provide a natural interpretation of the Laplacian product that is characterized via the classical box product on their bipartite representation graphs. This interpretation immediately explains the imaginary-like behavior of edges as well as the incidence-prism behavior. Comprehensive examples are then included to link the structure of the incidence matrices.

\subsection{Bipartite Interpretation via the Logical Functor}
\label{ssec:LapIsBipartite}

Recall from \cite[Theorem 3.47]{IH1} that there is a logical functor $\xymatrix{\cat{Q}\ar[r]^{\Upsilon} & \cat{R}}$, which admits both a left and a right adjoint.  Considering $\Upsilon$ and its adjoints deeply intertwine $\cat{R}$ and $\cat{Q}$, one would expect that it should connect their monoidal structure as well.  Unfortunately, none of them has satisfying monoidal behavior.

\begin{thm}[Laplacian product \& $\Upsilon$]
The logical functor $\Upsilon$ is not strong monoidal from $\left(\cat{Q},\vec{\Box},\vec{V}^{\diamond}(\{1\})\right)$ to either $\left(\cat{R},\blacksquare,\check{V}^{\diamond}(\{1\})\right)$ or $\left(\cat{R},\check{\Box},\check{V}^{\diamond}(\{1\})\right)$.  The adjoints $\Upsilon^{\star}$ and $\Upsilon^{\diamond}$ are not strong monoidal from $\left(\cat{R},\blacksquare,\check{V}^{\diamond}(\{1\})\right)$ or $\left(\cat{R},\check{\Box},\check{V}^{\diamond}(\{1\})\right)$ to $\left(\cat{Q},\vec{\Box},\vec{V}^{\diamond}(\{1\})\right)$.
\end{thm}

\begin{proof}

From direct calculation,
\begin{align*}
\Upsilon\left(\vec{V}^{\diamond}(\{1\})\right)&\cong\check{V}^{\diamond}(\{1\})\coprod\check{E}^{\diamond}(\{1\})\not\cong\check{V}^{\diamond}(\{1\}), \\ \Upsilon^{\star}\left(\check{V}^{\diamond}(\{1\})\right)&\cong\mathbb{0}_{\cat{Q}}\not\cong\vec{V}^{\diamond}(\{1\}).
\end{align*}
Thus, neither $\Upsilon$ nor $\Upsilon^{\star}$preserve the unit object.

The quivers $\Upsilon^{\diamond}\left(\mathbb{1}_{\cat{R}}\check{\Box}\mathbb{1}_{\cat{R}}\right)$, $\Upsilon^{\diamond}\left(\mathbb{1}_{\cat{R}}\blacksquare\mathbb{1}_{\cat{R}}\right)$, and $\Upsilon^{\diamond}\left(\mathbb{1}_{\cat{R}}\right)\vec{\Box}\Upsilon^{\diamond}\left(\mathbb{1}_{\cat{R}}\right)$ are drawn below.
\begin{center}
\begin{tabular}{|c|c|c|}
\hline
$\Upsilon^{\diamond}\left(\mathbb{1}_{\cat{R}}\check{\Box}\mathbb{1}_{\cat{R}}\right)$	&    $\Upsilon^{\diamond}\left(\mathbb{1}_{\cat{R}}\blacksquare\mathbb{1}_{\cat{R}}\right)$	&	$\Upsilon^{\diamond}\left(\mathbb{1}_{\cat{R}}\right)\vec{\Box}\Upsilon^{\diamond}\left(\mathbb{1}_{\cat{R}}\right)$\\
\hline
\includegraphics[scale=1]{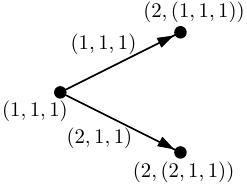}  &   \includegraphics[scale=1]{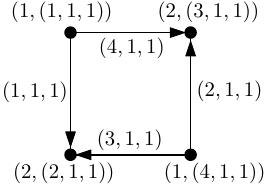}	&	\includegraphics[scale=1]{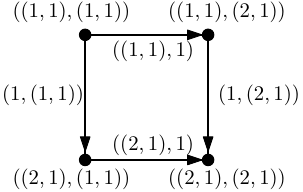}\\
\hline
\end{tabular}
\end{center}
\end{proof}

From the examples above, the only difference between $\Upsilon^{\diamond}(\cdot\blacksquare\cdot\cdot)$ and $\Upsilon^{\diamond}(\cdot)\vec{\Box}\Upsilon^{\diamond}(\cdot\cdot)$ is the direction of the edges.  Applying $U$ rectifies this, implying that $U\Upsilon^{\diamond}$ is a strong symmetric monoidal functor.  Furthermore, the following example emphasizes how $\blacksquare$ behaves far more coherently with $\Box$ under $U\Upsilon^{\diamond}$ than $\check{\Box}$.

\begin{ex}
Consider two paths of length $1$ in $\ob(\cat{R})$ and their products under $\check{\square}$ and $\blacksquare$. By sending each of them to their undirected bipartite equivalent graph via $U \Upsilon^{\diamond}$ we can examine the difference between the two products. Figure \ref{fig:RBoxWithUUpsilonDiamond} depicts the $\cat{R}$ box product (left) and its image under $U \Upsilon^{\diamond}$ (right). In the bipartite representation the vertices of this product are depicted as solid circles, while the edges appear as open circles. 

\begin{figure}[H]
    \centering
    \includegraphics{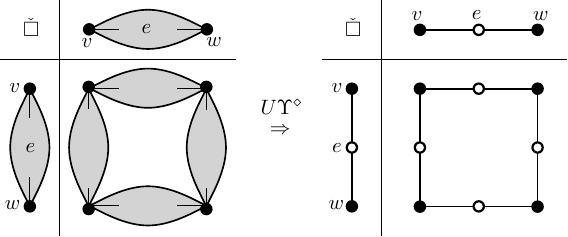}\\
    \caption{$\cat{R}$ box product under $U \Upsilon^{\diamond}$.}
    \label{fig:RBoxWithUUpsilonDiamond}
\end{figure}

Observe that $U \Upsilon^{\diamond}$ effectively doubles the length of a ``path'' as it translates it into $\cat{M}$ --- this is formalized in Corollary \ref{c:pathdouble}
\end{ex}

\begin{ex}
Now consider the same two paths of length $1$ under the Laplacian product. Figure \ref{fig:RLaplacianWithUUpsilonDiamond} depicts the Laplacian product (left) and its image under $U \Upsilon^{\diamond}$ (right). This is equivalent to taking the standard box product of the bipartite representation graphs. Again, the vertices in the product appear as solid circles, while the edges appear as open circles. The center $(e,e)$-vertex is not included in $\check{\square}$ but appears in $\blacksquare$.
\begin{figure}[H]
    \centering
    \includegraphics{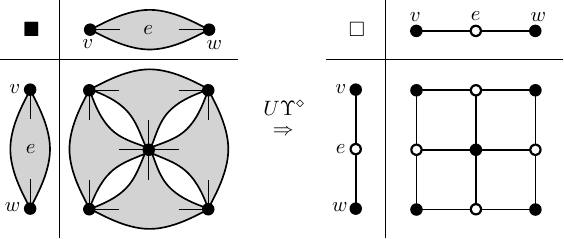}\\
    \caption{The Laplacian product of two paths of length $1$ treats $(e,e)$ as a vertex
    under $U \Upsilon^{\diamond}$.}
    \label{fig:RLaplacianWithUUpsilonDiamond}
\end{figure}
Dualizing one of the objects simply exchanges solid and open circles in the bipartite representation graph --- this is the ``Triforce of Duality'' in Figure \ref{triforce}. The Laplacian product's inclusion of $(e,e)$ pairs as vertices is analogous to the way the product of imaginary numbers are real. 
\end{ex}

Below is the monoidal structure for the composite functor, and the verification is routine.

\begin{defn}[Monoidal structure for $U\Upsilon^{\diamond}$]
For $G,H\in\ob(\cat{R})$, define\\ $\xymatrix{U\Upsilon^{\diamond}(G)\Box U\Upsilon^{\diamond}(H)\ar[r]^{\Psi_{G,H}} & U\Upsilon^{\diamond}(G\blacksquare H)}\in\cat{M}$ by
\begin{enumerate}
\item $V\left(\Psi_{G,H}\right)\left((1,v),(1,w)\right):=\left(1,(1,v,w)\right)$,
\item $V\left(\Psi_{G,H}\right)\left((2,e),(1,w)\right):=\left(2,(2,v,w)\right)$,
\item $V\left(\Psi_{G,H}\right)\left((1,v),(2,f)\right):=\left(2,(3,v,f)\right)$,
\item $V\left(\Psi_{G,H}\right)\left((2,e),(2,f)\right):=\left(1,(4,e,f)\right)$,

\item $E\left(\Psi_{G,H}\right)\left(1,i,(1,w)\right):=(1,i,w)$,
\item $E\left(\Psi_{G,H}\right)\left(1,i,(2,f)\right):=(2,i,f)$,
\item $E\left(\Psi_{G,H}\right)\left(2,(1,v),j\right):=(4,v,j)$,
\item $E\left(\Psi_{G,H}\right)\left(2,(2,e),j\right):=(3,e,j)$.
\end{enumerate}
Let $\xymatrix{V^{\diamond}(\{1\})\ar[r]^(0.4){\Psi_{\bullet}} & U\Upsilon^{\diamond}\left(\check{V}^{\diamond}(\{1\})\right)}\in\cat{M}$ be the unique map determined by $V\left(\Psi_{\bullet}\right)(1)=(1,1)$.
\end{defn}

\begin{thm}[Symmetric monoidal functor, $U\Upsilon^{\diamond}$]
Equipped with $\Psi$ and $\Psi_{\bullet}$, $U\Upsilon^{\diamond}$ is a strong symmetric monoidal functor from $\left(\cat{R},\blacksquare,\check{V}^{\diamond}(\{1\})\right)$ to $\left(\cat{M},\Box,V^{\diamond}(\{1\})\right)$.
\end{thm}

Since $U$ itself is monoidal, the following isomorphisms result, showing how $U$ and $\Upsilon^{\diamond}$ entangle with the box products $\vec{\Box}$, $\Box$, and $\blacksquare$.

\begin{cor}[Underlying Laplacian product]
For $G,H\in\ob(\cat{R})$, the following isomorphisms are natural.
\begin{align*}
U\left(\Upsilon^{\diamond}(G)\vec{\Box}\Upsilon^{\diamond}(H)\right)
\cong
U\Upsilon^{\diamond}(G) \Box U\Upsilon^{\diamond}(G)
\cong
U\Upsilon^{\diamond}(G \blacksquare H).
\end{align*}
\end{cor}

Moreover, the right adjoint $\Upsilon\vec{D}$ is a lax monoidal functor, but is sadly not strong.

\begin{cor}[Symmetric monoidal functor, $\Upsilon\vec{D}$]
The functor $\Upsilon\vec{D}$ is a lax symmetric monoidal functor from $\left(\cat{M},\Box,V^{\diamond}(\{1\})\right)$ to $\left(\cat{R},\blacksquare,\check{V}^{\diamond}(\{1\})\right)$, but is not strong.
\end{cor}

\begin{proof}
By \cite[p.\ 105]{lipman-hashimoto}, the strong monoidal structure of $U\Upsilon^{\diamond}$ yields a lax monoidal structure for $\Upsilon\vec{D}$.  Now, consider $\Upsilon\vec{D}\left(\mathbb{1}_{\cat{M}}\Box\mathbb{1}_{\cat{M}}\right)$ and $\Upsilon\vec{D}\left(\mathbb{1}_{\cat{M}}\right)\blacksquare\Upsilon\vec{D}\left(\mathbb{1}_{\cat{M}}\right)$.
\begin{center}
\includegraphics[scale=1]{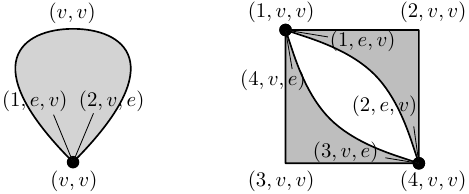}
\end{center}
\end{proof}

Furthermore, the monoidal structure of $U\Upsilon^{\diamond}$ deeply connects the traditional box exponential of $\cat{M}$ to the Laplacian exponential.

\begin{cor}[Laplacian \& box exponentials]
For all $G\in\ob(\cat{M})$ and $H\in\ob(\cat{R})$, the following natural isomorphism holds:  $\Upsilon\vec{D}\Del\left[U\Upsilon^{\diamond}(H),G\right]_{\beta}\cong\left[H,\Upsilon\vec{D}(G)\right]_{L}$.
\end{cor}

\begin{proof}
For $K\in\ob(\cat{R})$, one has
\begin{align*}
\cat{R}\left(K,\Upsilon\vec{D}\Del\left[U\Upsilon^{\diamond}(H),G\right]_{\beta}\right) &\cong		\cat{Q}\left(\Upsilon^{\diamond}(K),\vec{D}\Del\left[U\Upsilon^{\diamond}(H),G\right]_{\beta}\right) \\
&\cong\cat{M}\left(U\Upsilon^{\diamond}(K),\Del\left[U\Upsilon^{\diamond}(H),G\right]_{\beta}\right) \\
&\cong		\cat{H}\left(NU\Upsilon^{\diamond}(K),\left[U\Upsilon^{\diamond}(H),G\right]_{\beta}\right) \\
&\cong\cat{H}\left(U\Upsilon^{\diamond}(H)\Box NU\Upsilon^{\diamond}(K),G\right)\\
&=		\cat{M}\left(U\Upsilon^{\diamond}(H)\Box U\Upsilon^{\diamond}(K),G\right)
\cong\cat{M}\left(U\Upsilon^{\diamond}(H\blacksquare K),G\right)\\
&\cong		\cat{Q}\left(\Upsilon^{\diamond}(H\blacksquare K),\vec{D}(G)\right)
\cong\cat{R}\left(H\blacksquare K,\Upsilon\vec{D}(G)\right) \\
&\cong\cat{R}\left(K,\left[H,\Upsilon\vec{D}(G)\right]_{L}\right)
\end{align*}
\end{proof}

Effectively, paths in the incidence hypergraphs double in length as incidences are converted to edges in the undirected bipartite incidence graph. 

\begin{cor}[Paths \& box exponentials]
\label{c:pathdouble}
For all $G\in\ob(\cat{M})$ and $n\in\mathbb{N}$, $\left[\check{P}_{n/2},\Upsilon\vec{D}(G)\right]_{L}\cong\Upsilon\vec{D}\Del\left[P_{n},G\right]_{\beta}$.
\label{c:PathAndBox}
\end{cor}

\begin{proof}
From direct calculation, one can show $U\Upsilon^{\diamond}\left(\check{P}_{n/2}\right)\cong P_{n}$.  Thus,
\begin{align*}
\left[\check{P}_{n/2},\Upsilon\vec{D}(G)\right]_{L}
\cong\Upsilon\vec{D}\Del\left[U\Upsilon^{\diamond}\left(\check{P}_{n/2}\right),G\right]_{\beta}
\cong\Upsilon\vec{D}\Del\left[P_{n},G\right]_{\beta}.
\end{align*}
\end{proof}

\subsection{Examples and Why ``Laplacian'' Product?}

Due to its box-like structure, the terminology ``hom-box-product'' or ``complete box product'' seems just as valid. Here we use Theorem \ref{t:WWT} as motivation to provide examples of how to reclaim the signless Laplacian by evaluating the Laplacian exponential at paths; incidence orientations can be applied afterwards for the general oriented hypergraphic Laplacian. We have shown that
\begin{align*}
\check{V}[\check{P}_{k/2},G]_{L}&=\cat{R}(\check{P}_{k/2},G), \\
\check{E}[\check{P}_{k/2},G]_{L}&=\cat{R}\left(\check{P}_{k/2}^{\#},G\right)\cong\cat{R}\left(\check{P}_{k/2},G^{\#}\right), \\
I\left[\check{P}_{k/2},G\right]_{L}
&=\cat{R}\left(\check{P}_{k/2}\blacksquare I^{\diamond}\left(\{1\}\right),G\right).
\end{align*}

In general, the incidences of the Laplacian exponential are incidence-prism mappings. However, the incidences $I\left[\check{P}_{k/2},G\right]_{L}
=\cat{R}\left(\check{P}_{k/2}\blacksquare I^{\diamond}\left(\{1\}\right),G\right)$ are mappings of the an incidence-ladder graph, in which one side is the dual of the other, into $G$ --- as seen by Figures \ref{fig:RLaplacianProd1} and \ref{fig:RLaplacianProd2}. Each new rung creates a new digon to map that effectively searches for interlocking $2 \times 2$ minors. Moreover, mappings onto a single incidence, adjacency, or co-adjacency are allowed, thus creating a morphism connection between the entry (incidence map), row (adjacency map), and column (co-adjacency map). 

Consider $I\left[\check{P}_{1/2},G\right]_{L} =\cat{R}\left(\check{P}_{1/2}\blacksquare I^{\diamond}\left(\{1\}\right),G\right)=\cat{R}\left(\check{P}_{1/2}\blacksquare \check{P}_{1/2},G\right)$ and a digon-free incidence-simple incidence hypergraph $G$. The vertices and edges correspond to the entries of the complete incidence matrix $\overline{\mathbf{H}}_{G}$, while the incidences are determined by the mappings of the digon from Figure \ref{fig:RLaplacianProd1}. There are three possible maps of the digon: (1) a vertex-to-vertex backstep (or edge-to-edge co-backstep) to determine entry of ${\mathbf{H}}_{G}$ (or ${\mathbf{H}}^T_{G}$); (2) a vertex-to-vertex adjacency; and (3) an edge-to-edge co-adjacency. The first map identifies a specific incidence in $G$; which for $[\check{P}_{1/2},G]_{L}$ is also naturally associated to a vertex --- this is regarded as the vertex representing the location in the incidence matrix, that is occupied by a value $1$ for the backstep-incidence in the $(v,e)$ and $(e,v)$ positions by duality. Effectively, the digon was searching for a $2 \times 2$ minor but has collapsed onto a single entry. Now consider the second and third map types that include incidence $i$ as its first incidence; these produce all adjacencies and co-adjacencies that contain that incidence --- again, these are collapsed $2 \times 2$ minors onto $2$ entries in a row/column search.

\begin{ex}
Consider the incidence graph $G$ and its dual $G^{\#}$ in Figure \ref{fig:LapExpEx1}, with incidence matrices

\begin{align*}
\mathbf{H}_{G}&=\left[
\begin{array}{ccccc}
1 & 0 & 0 & 1 & 1 \\ 
1 & 1 & 0 & 0 & 0 \\ 
0 & 1 & 1 & 0 & 1 \\
0 & 0 & 1 & 1 & 0 
\end{array}
\right], & \mathbf{H}_{G^{\#}}=\mathbf{H}^{T}_{G}=\left[
\begin{array}{cccc}
1 & 1 & 0 & 0 \\ 
0 & 1 & 1 & 0 \\ 
0 & 0 & 1 & 1 \\
1 & 0 & 0 & 1 \\
1 & 0 & 1 & 0 
\end{array}
\right].
\end{align*}

To calculate the vertices and edges of $[\check{P}_{1/2},G]_{L}$ consider the mapping of a single path of length $1/2$ into $G$ and $G^{\#}$, respectively. 
\begin{figure}[H]
    \centering
    \includegraphics{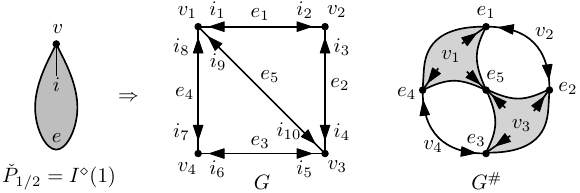}\\
    \caption{An extroverted oriented hypergraph $G$ and its dual $G^{\#}$}
    \label{fig:LapExpEx1}
\end{figure}

Clearly, maps of $\check{P}_{1/2}$ are uniquely determined by the image of incidence $i$, of which there are exactly $10$. Thus, there are $10$ vertices in $[\check{P}_{1/2},G]_{L}$, and the incidence matrix $\mathbf{H}_{G}$ can be recovered by the vertex-edge image $(v_j,e_k)$ corresponding to an entry of $1$ in the $(j,k)$ position of $\mathbf{H}_{G}$. In Figure \ref{fig:LapExpEx2AB} (left) the vertices are placed in a $\lvert V \rvert \times \lvert E \rvert$ grid, corresponding to the non-zero entries of $\mathbf{H}_{G}$.

\begin{figure}[H]
    \centering
    \includegraphics{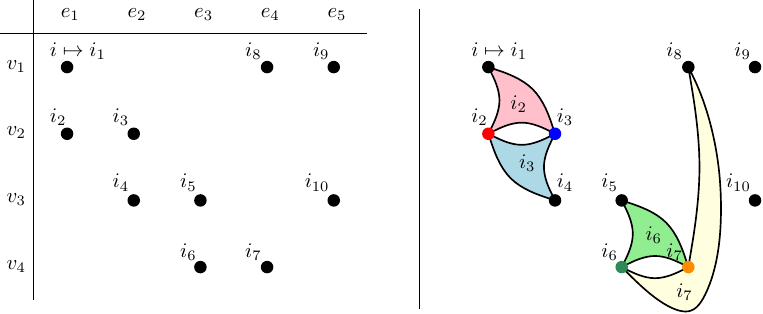}\\
    \caption{Left: The vertices of the incidence hypergraph $[\check{P}_{1/2},G]_{L}$ are the non-zero positions in the incidence matrix. Right: The vertices for rows $2$ and $4$ are colored, and their row/column pairing for each position determine the edge and incidences (corresponding edges colored).}
    \label{fig:LapExpEx2AB}
\end{figure}

Since $\check{P}_{1/2}$ is self-dual there are also $10$ edges. These edges are connected to the vertices by the incidences determined by the images of the digon in Figure \ref{fig:RLaplacianProd1}. The incidences in the edges are determined by edge $i \mapsto i_{\ell}$ is incident to all the vertices in the row and column of vertex $i \mapsto i_{\ell}$. To see this, consider Figure \ref{fig:LapExpEx2AB} (right). Both the vertex and edge obtained by the map $i \mapsto i_{2}$ are colored red, while the backstep incidence map that corresponds the the backstep $(i_2,i_2)$ is the incidence between the red vertex and red edge. The other adjacency and co-adjacency digon maps reach the vertices that correspond to non-zero entries in the matrix. This argument is repeated for the other vertices. Figure \ref{fig:LapExpEx2CD} calculates the edges and incidences for rows $1$ and $3$, respectively. Again, each colored vertex has an edge corresponding to the row/column pair, with incidences where the non-zero entries are located. 

\begin{figure}[H]
    \centering
    \includegraphics{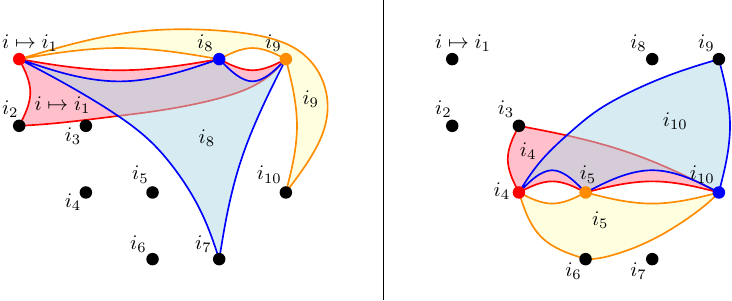}\\
    \caption{Left: The vertices for row $1$ are colored, and their row/column pairing for each position determine the edge and incidences (corresponding edges colored).  Right: The vertices for row $3$ are colored, and their row/column pairing for each position determine the edge and incidences (corresponding edges colored).}
    \label{fig:LapExpEx2CD}
\end{figure}

The entire $[\check{P}_{1/2},G]_{L}$ appears in Figure \ref{fig:LapExpEx2E}.

\begin{figure}[H]
    \centering
    \includegraphics{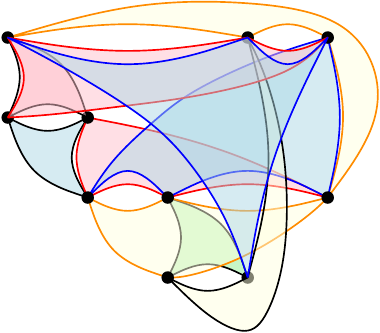}\\
    \caption{The incidence hypergraph $[\check{P}_{1/2},G]_{L}$.}
    \label{fig:LapExpEx2E}
\end{figure}

The matrix $\mathbf{H}^T_{G}$ is found dually by interchanging vertices and edges; or the reader may simply ``transpose'' $[\check{P}_{1/2},G]_{L}$ in Figure \ref{fig:LapExpEx2E}.
\label{e:LastExample}
\end{ex}

If there are multiple incidences, then there are additional incidences for each loop adjacency. If there are digons, there are additional incidences for each digon map ($2 \times 2$ minor).  In Example \ref{e:ReallyLastExample} we show that multiple incidences and digons extend the row and column sampling.

\begin{ex}

Consider the incidence hypergraph $G$ in Figure \ref{fig:LapExpEx2} (again with constant orientation) with incidence matrix
\begin{align*}
\mathbf{H}_{G}&=\left[
\begin{array}{cc}
1 & 2  \\ 
1 & 1  \\ 
1 & 0  \\
\end{array}
\right].
\end{align*}
$G$ is depicted in Figure \ref{fig:LapExpEx2} (right) with its incidences in position with the incidence matrix entries.
\begin{figure}[H]
    \centering
    \includegraphics{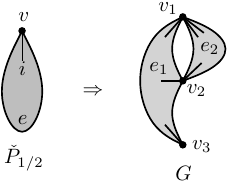}\\
    \caption{An incidence hypergraph $G$}
    \label{fig:LapExpEx2}
\end{figure}

Again, the vertices and edges of $[\check{P}_{1/2},G]_{L}$ are the incidences of $G$, thus there are $6$ vertices and $6$ edges, as parallel incidences are counted separately. Figure \ref{fig:LapExpEx2Part1} (left) shows the vertices of $[\check{P}_{1/2},G]_{L}$ arranged into the ``incidence matrix.'' 
\begin{figure}[H]
    \centering
    \includegraphics{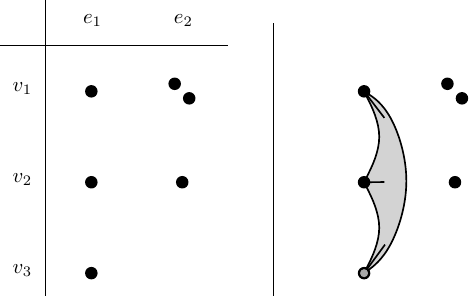}\\
    \caption{Left: The vertices of the incidence hypergraph $[\check{P}_{1/2},G]_{L}$ correspond to the non-zero entries of the incidence matrix. Right: The edge from the $(v_3,e_1)$ incidence reaches all the incidences in its corresponding row and column in the incidence matrix.}
    \label{fig:LapExpEx2Part1}
\end{figure}

Again, the digon map produces the incidences of $[\check{P}_{1/2},G]_{L}$. However, the $(v_3,e_1)$ incidence can only reach the incidences within $e_1$ to form an edge --- this can be interpreted as starting incidence matrix value $1$ in the $(v_3,e_1)$ position and searching its row and column for non-zero entries. This gives rise to the edge in Figure \ref{fig:LapExpEx2Part1} (left).

The remaining incidences of $G$ are either in the digon and/or are part of a parallel incidence. The parallel incidences causes multi-sampling of the row/column when it appears in a digon embedding. However, the digon in $G$ will cause an additional incidence in the edge when two non-zero entries in a row and column ``triangulate'' at the non-zero entry --- effectively finding a $2 \times 2$ minor with all non-zero entries in the incidence matrix (up to multiplicity of incidences). 

Consider the $(v_2,e_1)$ entry in the incidence matrix and its row and column in Figure \ref{fig:LapExpEx2MatrixSearch}. The backstep mapping attaches the $(v_2,e_1)$ edge to the $(v_2,e_1)$ vertex, while the adjacency and co-adjacency maps search the columns and rows for non-zero entries. These non-zero entries then search for $2 \times 2$ minor with all non-zero entries via the remaining digon mapping. There are two mappings to the $(v_1,e_2)$ position as there are are two parallel incidences, while there are no mappings to the $(v_3,e_2)$ position.

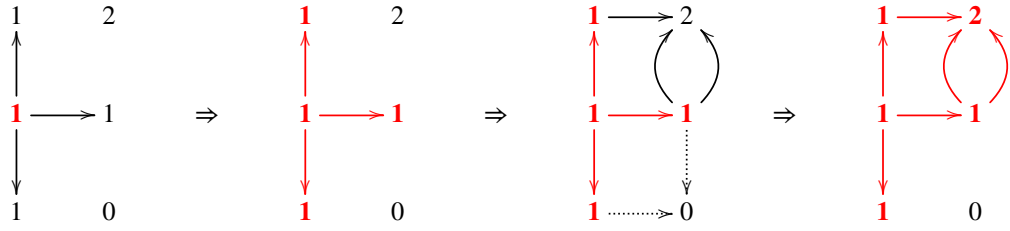
\begin{figure}[H]
\begin{align*}
\xymatrix{
1 &	2 & &  \textcolor{red}{\textbf{1}} & 2 & & \textcolor{red}{\textbf{1}} \ar[r] & 2 & & \textcolor{red}{\textbf{1}} \ar@[red][r] & \textcolor{red}{\textbf{2}}\\
\textcolor{red}{\textbf{1}} \ar[u] \ar[d] \ar[r]& 1  & \Rightarrow &  \textcolor{red}{\textbf{1}} \ar@[red][u] \ar@[red][d] \ar@[red][r] & \textcolor{red}{\textbf{1}} & \Rightarrow & \textcolor{red}{\textbf{1}} \ar@[red][u] \ar@[red][d] \ar@[red][r] & \textcolor{red}{\textbf{1}} \ar@/^1pc/@{->}[u] \ar@/_1pc/@{->}[u] \ar@{.>}[d] & \Rightarrow & \textcolor{red}{\textbf{1}} \ar@[red][u] \ar@[red][d] \ar@[red][r] & \textcolor{red}{\textbf{1}} \ar@[red]@/^1pc/@{->}[u] \ar@[red]@/_1pc/@{->}[u] \\
1 & 0 & &  \textcolor{red}{\textbf{1}} & 0 & & \textcolor{red}{\textbf{1}} \ar@{.>}[r] & 0 & & \textcolor{red}{\textbf{1}} & 0 \\
}
\end{align*}
\caption{Digon mappings to produce incidences of $[\check{P}_{1/2},G]_{L}$ are row/column searches that form $2 \times 2$ minor grids; $(v_2,e_1)$ shown.}
\label{fig:LapExpEx2MatrixSearch}
\end{figure}
The edge containing the incidence from Figure \ref{fig:LapExpEx2MatrixSearch} is the first edge in Figure \ref{fig:LapExpEx2Part2}. The remaining edges are determined similarly. 
\begin{figure}[H]
    \centering
    \includegraphics{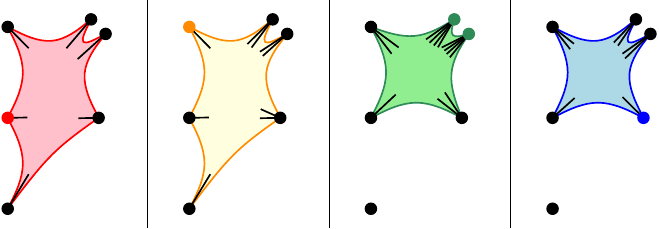}\\
    \caption{The edges of $[\check{P}_{1/2},G]_{L}$ using the digon or parallel incidence.}
    \label{fig:LapExpEx2Part2}
\end{figure}
The third edge of Figure \ref{fig:LapExpEx2Part2} appears twice, one for each parallel incidence. Thus, parallel incidences produce parallel edges in $[\check{P}_{1/2},G]_{L}$.
\label{e:ReallyLastExample}
\end{ex}

We conclude with the note that the vertices and edges of $[\check{P}_{k/2},G]_{L}$ are naturally labeled (with multiplicity) by the entries locations in $\overline{\mathbf{H}}_{G}^{k}$, and when $k$ is even they correspond to powers of the Laplacian. The incidences are determined by mappings of the incidence-ladder $\check{P}_{k/2}\blacksquare \check{P}_{1/2}$ into $G$ that provide the insight on the minor structure of the incidence matrix. It is also worth asking if replacing ``paths'' with another class of graphs will provide ``Laplacian-type'' results or ways to study graph-like structures. Evaluations of the Laplacian exponential at $k$-cycles, $[\check{C}_{k},G]_{L}$, when restricted to monic maps determine locally-hamiltonicity of vertices and edges, and local-eulerianness of the incidences --- this is immediate from the bipartite representation. As to the answer on why restrict to monic maps? We recall that relaxation of the monicness of path embeddings unify the study of adjacency and Laplacian matrices \cite{OHHar,IH2,AH1}; imposing monic conditions seems as natural.

\begin{appendices}

\section{Incidence Hypergraphs}
\label{IHdefn}

Formally, an incidence hypergraph (from \cite[p.\ 17]{IH1}) is defined as follows:
Let $\cat{D}$ be the finite category
\begin{align*}
\xymatrix{
0	&	2\ar[l]_{y}\ar[r]^{z}	&	1\\
}
\end{align*}
and the category of incidence hypergraphs is $\cat{R}:=\Set^{\cat{D}}$ with evaluation functors 
\begin{align*}
\xymatrix{\Set & \cat{R}\ar@/_/[l]_(0.4){\check{V}}\ar@/^/[l]^(0.4){\check{E}}\ar[r]^(0.4){I} & \Set} 
\end{align*}
at $0$, $1$, and $2$, respectively.  An object $G$ of $\cat{R}$ consists of the following:  a set $\check{V}(G)$, a set $\check{E}(G)$, a set $I(G)$, a function $\varsigma_{G}:I(G)\to\check{V}(G)$, and a function $\omega_{G}:I(G)\to\check{E}(G)$. Note that the incidence function $\iota_{G}: I(G)\rightarrow\check{V}(G)\times\check{E}(G)$ used in \cite{OHSachs,OHMTT} is uniquely determined by the diagram below, where $\pi_{\check{V}(G)}$ and $\pi_{\check{E}(G)}$ are the canonical projections.
\begin{align*}
\xymatrix{
&   &   I(G)\ar[dll]_{\varsigma_{G}}\ar[drr]^{\omega_{G}}\ar@{..>}[d]^{\exists!\iota_{G}}\\
\check{V}(G)    &   &   \check{V}(G)\times\check{E}(G)\ar[ll]^{\pi_{\check{V}(G)}}\ar[rr]_{\pi_{\check{E}(G)}}  &    &   \check{E}(G)\\
}
\end{align*}
Note from \cite{IH1} that the single incidence $1$-edge $I^{\diamond}(\{1\})$ is both the terminal object as well as the non-trivial generator of the category --- the other generators being the isolated vertex $\check{V}^{\diamond}(\{1\})$ and loose edge $\check{E}^{\diamond}(\{1\})$.

\section{Structure Maps}
\label{StrMaps}
The structure maps for each symmetric monoidal box product are included here for completeness. The verification of the necessary identities is tedious but routine.

\subsection{Quivers}
\begin{defn}[Structure maps]
For quivers $M$, $P$, and $Q$, define the following structure maps:
\begin{enumerate}
\item $\xymatrix{Q\vec{\Box}\vec{V}^{\diamond}(\{1\})\ar[r]^(0.7){\vec{r}_{Q}} & Q}$ by $\vec{V}\left(\vec{r}_{Q}\right)(v,1):=v$, $\vec{E}\left(\vec{r}_{Q}\right)(1,e,1):=e$;
\item $\xymatrix{\vec{V}^{\diamond}(\{1\})\vec{\Box}Q\ar[r]^(0.7){\vec{\ell}_{Q}} & Q}$ by $\vec{V}\left(\vec{\ell}_{Q}\right)(1,v):=v$, $\vec{E}\left(\vec{\ell}_{Q}\right)(2,1,e):=e$;
\item $\xymatrix{Q\vec{\Box}P\ar[r]^{\vec{c}_{Q,P}} & P\vec{\Box}Q}$ by $\vec{V}\left(\vec{c}_{Q,P}\right)(v,w):=(w,v)$, $\vec{E}\left(\vec{c}_{Q,P}\right)(n,x,y):=(3-n,y,x)$;
\item $\xymatrix{\left(Q\vec{\Box}P\right)\vec{\Box}M\ar[r]^{\vec{a}_{Q,P,M}} & Q\vec{\Box}\left(P\vec{\Box}M\right)}$ by
\begin{itemize}
\item $\vec{V}\left(\vec{a}_{Q,P,M}\right)((v,w),u):=(v,(w,u))$,
\item $\vec{E}\left(\vec{a}_{Q,P,M}\right)(1,(1,e,w),u):=(1,e,(w,u))$,
\item $\vec{E}\left(\vec{a}_{Q,P,M}\right)(1,(2,v,f),u):=(2,v,(1,f,u))$,
\item $\vec{E}\left(\vec{a}_{Q,P,M}\right)(2,(v,w),g):=(2,v,(2,w,g))$.
\end{itemize}
\end{enumerate}
\end{defn}

\subsection{Set systems and Multigraphs}
\label{StrMapsH}
\begin{defn}[Structure maps]
For set-system hypergraphs $G$, $H$, and $K$, define the following structure maps:
\begin{enumerate}
\item $\xymatrix{G\Box V^{\diamond}(\{1\})\ar[r]^(0.7){r_{G}} & G}$ by $V\left(r_{G}\right)(v,1):=v$, $E\left(r_{G}\right)(1,e,1):=e$;
\item $\xymatrix{V^{\diamond}(\{1\})\Box G\ar[r]^(0.7){\ell_{G}} & G}$ by $V\left(\ell_{G}\right)(1,v):=v$, $E\left(\ell_{G}\right)(2,1,e):=e$;
\item $\xymatrix{G\Box H\ar[r]^{c_{G,H}} & H\Box G}$ by $V\left(c_{G,H}\right)(v,w):=(w,v)$, $E\left(c_{G,H}\right)(n,x,y):=(3-n,y,x)$;
\item $\xymatrix{\left(G\Box H\right)\Box K\ar[r]^{a_{G,H,K}} & G\Box\left(H \Box K\right)}$ by
\begin{itemize}
\item $V\left(a_{G,H,K}\right)((v,w),u):=(v,(w,u))$,
\item $E\left(a_{G,H,K}\right)(1,(1,e,w),u):=(1,e,(w,u))$,
\item $E\left(a_{G,H,K}\right)(1,(2,v,f),u):=(2,v,(1,f,u))$,
\item $E\left(a_{G,H,K}\right)(2,(v,w),g):=(2,v,(2,w,g))$.
\end{itemize}
\end{enumerate}
\end{defn}

\subsection{Incidence Hypergraphs}
\label{StrMapsR}
\begin{defn}[Structure maps]
For incidence hypergraphs $G$, $H$, and $K$, define the following structure maps:
\begin{enumerate}
\item $\xymatrix{G\check{\Box}\check{V}^{\diamond}(\{1\})\ar[r]^(0.7){\check{r}_{G}} & G}$ by $\check{V}\left(\check{r}_{G}\right)(v,1):=v$, $\check{E}\left(\check{r}_{G}\right)(1,e,1):=e$, $I\left(\check{r}_{G}\right)(1,i,1):=i$;
\item $\xymatrix{\check{V}^{\diamond}(\{1\})\check{\Box}G\ar[r]^(0.7){\check{\ell}_{G}} & G}$ by $\check{V}\left(\check{\ell}_{G}\right)(1,v):=v$, $\check{E}\left(\check{\ell}_{G}\right)(2,1,e):=e$, $I\left(\check{\ell}_{G}\right)(2,1,i):=i$;
\item $\xymatrix{G\check{\Box}H\ar[r]^{\check{c}_{G,H}} & H\check{\Box}G}$ by $\check{V}\left(\check{c}_{G,H}\right)(x,y):=(y,x)$, $\check{E}\left(\check{c}_{G,H}\right)(n,x,y):=(3-n,y,x)$, $I\left(\check{c}_{G,H}\right)(n,x,y):=(3-n,y,x)$;
\item $\xymatrix{\left(G\check{\Box}H\right)\check{\Box}K\ar[r]^{\check{a}_{G,H,K}} & G\check{\Box}\left(H\check{\Box}K\right)}$ by
\begin{itemize}
\item $\check{V}\left(\check{a}_{G,H,K}\right)((v,w),u):=(v,(w,u))$,

\item $\check{E}\left(\check{a}_{G,H,K}\right)(1,(1,e,w),u):=(1,e,(w,u))$,
\item $\check{E}\left(\check{a}_{G,H,K}\right)(1,(2,v,f),u):=(2,v,(1,f,u))$,
\item $\check{E}\left(\check{a}_{G,H,K}\right)(2,(v,w),g):=(2,v,(2,w,g))$,

\item $I\left(\check{a}_{G,H,K}\right)(1,(1,i,w),u):=(1,i,(w,u))$,
\item $I\left(\check{a}_{G,H,K}\right)(1,(2,v,j),u):=(2,v,(1,j,u))$,
\item $I\left(\check{a}_{G,H,K}\right)(2,(v,w),k):=(2,v,(2,w,k))$.
\end{itemize}
\end{enumerate}
\end{defn}

\subsection{Laplacian Product}
\label{StrMapsL}
\begin{defn}[Structure maps]
For incidence hypergraphs $G$, $H$, and $K$, define the following structure maps:
\begin{enumerate}
\item $\xymatrix{G\blacksquare\check{V}^{\diamond}(\{1\})\ar[r]^(0.7){\check{\rho}_{G}} & G}$ by $\check{V}\left(\check{\rho}_{G}\right)(1,v,1):=v$, $\check{E}\left(\check{\rho}_{G}\right)(2,e,1):=e$, $I\left(\check{\rho}_{G}\right)(1,i,1):=i$;
\item $\xymatrix{\check{V}^{\diamond}(\{1\})\blacksquare G\ar[r]^(0.7){\check{\lambda}_{G}} & G}$ by $\check{V}\left(\check{\lambda}_{G}\right)(1,1,v):=v$, $\check{E}\left(\check{\lambda}_{G}\right)(3,1,e):=e$, $I\left(\check{\lambda}_{G}\right)(4,1,i):=i$;
\item $\xymatrix{G\blacksquare H\ar[r]^{\check{\gamma}_{G,H}} & H\blacksquare G}$ by $\check{V}\left(\check{\gamma}_{G,H}\right)(n,x,y):=(n,y,x)$, $\check{E}\left(\check{\gamma}_{G,H}\right)(n,x,y):=(5-n,y,x)$, $I\left(\check{\gamma}_{G,H}\right)(n,x,y):=(5-n,y,x)$;
\item $\xymatrix{\left(G\blacksquare H\right)\blacksquare K\ar[r]^{\check{\alpha}_{G,H,K}} & G\blacksquare\left(H\blacksquare K\right)}$ by
\begin{multicols}{2}
\begin{itemize}
\item $\check{V}\left(\check{\alpha}_{G,H,K}\right)(1,(1,v,w),u):= \\ (1,v,(1,w,u))$,
\item $\check{V}\left(\check{\alpha}_{G,H,K}\right)(1,(4,e,f),u):= \\ (4,e,(2,f,u))$,
\item $\check{V}\left(\check{\alpha}_{G,H,K}\right)(4,(2,e,w),g):= \\ (4,e,(3,w,g))$,
\item $\check{V}\left(\check{\alpha}_{G,H,K}\right)(4,(3,v,f),g):= \\ (1,v,(4,f,g))$,

\item $\check{E}\left(\check{\alpha}_{G,H,K}\right)(2,(2,e,w),u):= \\ (2,e,(1,w,u))$,
\item $\check{E}\left(\check{\alpha}_{G,H,K}\right)(2,(3,v,f),u):= \\ (3,v,(2,f,u))$,
\item $\check{E}\left(\check{\alpha}_{G,H,K}\right)(3,(1,v,w),g):= \\ (3,v,(3,w,g))$,
\item $\check{E}\left(\check{\alpha}_{G,H,K}\right)(3,(4,e,f),g):= \\ (2,e,(4,f,g))$,

\item $I\left(\check{\alpha}_{G,H,K}\right)(1,(1,i,w),u):= \\ (1,i,(1,w,u))$,
\item $I\left(\check{\alpha}_{G,H,K}\right)(1,(2,i,f),u):= \\ (2,i,(2,f,u))$,
\item $I\left(\check{\alpha}_{G,H,K}\right)(1,(3,e,j),u):= \\ (3,e,(3,j,u))$,
\item $I\left(\check{\alpha}_{G,H,K}\right)(1,(4,v,j),u):= \\ (4,v,(4,j,u))$,

\item $I\left(\check{\alpha}_{G,H,K}\right)(2,(1,i,w),g):= \\ (2,i,(3,w,g))$,
\item $I\left(\check{\alpha}_{G,H,K}\right)(2,(2,i,f),g):= \\ (1,i,(4,f,g))$,
\item $I\left(\check{\alpha}_{G,H,K}\right)(2,(3,e,j),g):= \\ (3,e,(2,j,g))$,
\item $I\left(\check{\alpha}_{G,H,K}\right)(2,(4,v,j),g):= \\ (4,v,(2,j,g))$,

\item $I\left(\check{\alpha}_{G,H,K}\right)(3,(2,e,w),k):= \\ (3,e,(4,w,k))$,
\item $I\left(\check{\alpha}_{G,H,K}\right)(3,(3,v,f),k):= \\ (4,v,(3,f,k))$,

\item $I\left(\check{\alpha}_{G,H,K}\right)(4,(1,v,w),k):= \\ (4,v,(4,w,k))$,
\item $I\left(\check{\alpha}_{G,H,K}\right)(4,(4,e,f),k):= \\ (3,e,(3,f,k))$.
\end{itemize}
\end{multicols}
\end{enumerate}
\end{defn}

\section{Glossary}
\label{glossary}

\begin{table}[ht]
\centering
\renewcommand{\arraystretch}{1}
\begin{tabular}{|l|l|}
\hline
Expression & Category Note \\
\hline\hline
$\Set$ & Category of sets (topos). \\
$\cat{Q}$ & Category of quivers (topos). \\
$\cat{R}$ & Category of incidence hypergraphs (topos). \\
$\cat{M}$ & Category of multigraphs.\\
$\cat{H}$ & Category of set system hypergraphs.\\
$\cat{E}$ & The finite category \xymatrix{1\ar@/^/[r]^{s}\ar@/_/[r]_{t}	&	0\\}; $\cat{Q}=\Set^{\cat{E}}$. \\
$\cat{D}$ & The finite category \xymatrix{0	&	2\ar[l]_{y}\ar[r]^{z}	&	1\\}; $\cat{R}=\Set^{\cat{D}}$. \\
$\cat{A}(G,H)$ & Homomorphisms from $G$ to $H$ in category $\cat{A}$.\\
\hline
\end{tabular}
\caption{Glossary of categorical names.}
\label{t:library1}
\end{table}

\begin{table}[ht]
\centering
\renewcommand{\arraystretch}{1}
\begin{tabular}{|l|l|}
\hline
Decorations & Note \\
\hline\hline
$\vec{\cdot}$ & Pertaining to $\cat{Q}$. Vertices $\vec{V}$ and edges $\vec{E}$. \\
$\check{\cdot}$ & Pertaining to $\cat{R}$. Vertices $\check{V}$ and edges $\check{E}$. \\
$\cdot$ & Pertaining to $\cat{M}$ and $\cat{H}$. Vertices $V$ and edges $E$. \\
$\cdot^{\star}$ & Right adjoint. Generally maximal; completed graph $V^{\star}$.\\
$\cdot^{\diamond}$ & Left adjoint. Generally disjoint; isolated vertices $V^{\diamond}$; a single vertex $V^{\diamond}(\{1\})$.\\
$\cdot^{\#}$ & Incidence duality in $\cat{R}$.\\
\hline
\end{tabular}
\caption{Glossary of functor decorations.}
\label{t:library2}
\end{table}

\begin{table}[ht]
\centering
\renewcommand{\arraystretch}{1}
\begin{tabular}{|l|l|}
\hline
Functors & Evaluation Note \\
\hline\hline
$U$ & Undirects directed edges; $\cat{Q} \rightarrow \cat{M}$. \\
$\vec{D}$ & Forms equivalent digraph; $\cat{M} \rightarrow \cat{Q}$. \\
$N$ & Inclusion of a graph as a set system; $\cat{M} \rightarrow \cat{H}$.\\
$\Del$ & Deletes any edges of size greater than 2 or less than 1; $\cat{H} \rightarrow \cat{M}$. \\
$\mathcal{I}$ & Add incidences to hyperedges; $\cat{H} \rightarrow \cat{R}$. \\
$\mathscr{F}$ & Forget incidences; $\cat{R} \rightarrow \cat{H}$ (not a functor). \\
$I$ & Incidence set; $\cat{R} \rightarrow \Set$. \\
$I^{\diamond}(\{1\})$ & A single incidence with vertex and edge ($1$-edge). \\
$\mathcal{P}$ & Powerset; $\Set \rightarrow \Set$. \\
$\Upsilon $  &  Converts directed edges to incidences; $\cat{Q} \rightarrow \cat{R}$, (logical functor). \\ 
$\Upsilon ^{\diamond}$  & Bipartite incidence digraph ($v\to e$). \\ 
$U \Upsilon ^{\diamond}$  & Bipartite incidence graph. \\ 
$Y_{\cat{A}}$ & Yoneda embedding into $\cat{A}$.\\
\hline
\end{tabular}
\caption{Glossary of functors.}
\label{t:library3}
\end{table}

\begin{table}[ht]
\centering
\renewcommand{\arraystretch}{1}
\begin{tabular}{|l|l|}
\hline
Operations & Notes \\
\hline\hline
$\vec{\Box}$ & Box product in $\cat{Q}$. \\
$\check{\Box}$ & Box product in $\cat{R}$. \\
${\Box}$ & Box product in $\cat{M}$ and $\cat{H}$. \\
$\blacksquare$ & Laplacian product in $\cat{R}$. \\
$\cdot\mathrm{ev}_{X_{2}}^{X_{1}}$ & Evaluation map to form the box exponential. \\
$\left[X_{1},X_{2}\right]_{B}$ & Box exponential in $\cat{Q}$. \\
$\left[X_{1},X_{2}\right]_{V}$ & Box exponential in $\cat{R}$. \\
$\left[X_{1},X_{2}\right]_{\beta}$ & Box exponential in $\cat{M}$ and $\cat{H}$. \\
$\left[X_{1},X_{2}\right]_{L}$ & Laplacian exponential in $\cat{R}$. \\
\hline
\end{tabular}
\caption{Glossary of operations.}
\label{t:library4}
\end{table}

\end{appendices}



\end{document}